\documentclass[11pt,a4paper]{article}
\usepackage{amsmath}
\usepackage{amsfonts}
\usepackage{amssymb}
\usepackage{centernot}
\usepackage[margin=3cm]{geometry}
\usepackage[colorlinks=true,linkcolor=blue,citecolor=blue]{hyperref}
\usepackage{enumitem}
\usepackage{graphicx,fancybox}
\usepackage{tikz}
\usepackage{subfigure}
\usetikzlibrary{positioning}
\usepackage{tkz-berge}
\usepackage{hhline}

\usepackage{amsthm}
\newtheorem{fact}{Fact}[section]
\newtheorem{example}{Example}[section]

\newcommand{\tr}{\operatorname{tr}}
\newcommand{\round}{\operatorname{round}}
\newcommand{\Wd}{\operatorname{Wd}}
\newcommand{\Hi}{\mathcal{H}}
\newcommand{\R}{\mathbb{R}}

\newcounter{row}
\newcounter{col}
\newcommand\setrow[9]{
	\setcounter{col}{1}
	\foreach \n in {#1, #2, #3, #4, #5, #6, #7, #8, #9} {
		\edef\x{\value{col} - 0.5}
		\edef\y{9.5 - \value{row}}
		\node[anchor=center] at (\x, \y) {\n};
		\stepcounter{col}
	}
	\stepcounter{row}
}

\DeclareFontFamily{U}{mathx}{\hyphenchar\font45}
\DeclareFontShape{U}{mathx}{m}{n}{
      <5> <6> <7> <8> <9> <10>
      <10.95> <12> <14.4> <17.28> <20.74> <24.88>
      mathx10
      }{}
\DeclareSymbolFont{mathx}{U}{mathx}{m}{n}
\DeclareFontSubstitution{U}{mathx}{m}{n}
\DeclareMathAccent{\widecheck}{0}{mathx}{"71}

\title{Solving graph coloring  problems\\ with the Douglas--Rachford algorithm}

\author{Francisco J. Arag\'on Artacho\thanks{Department of Mathematics,
University of Alicante, \textsc{Spain}. e-mail:~\url{francisco.aragon@ua.es}}
        \and Rub\'en Campoy\thanks{Department of Mathematics,
University of Alicante, \textsc{Spain}. e-mail:~\url{ruben.campoy@ua.es}}
}
\date{\smallskip \bf \small Dedicated to the memory of Jonathan Michael Borwein}

\begin{document}

\maketitle

\begin{abstract}
We present the Douglas--Rachford algorithm as a successful heuristic for solving graph coloring problems. Given a set of colors, these type of problems consist in assigning a color to each node of a graph, in such a way that every pair of adjacent nodes are assigned with different colors. We formulate the graph coloring problem as an appropriate feasibility problem that can be effectively solved by the Douglas--Rachford algorithm, despite the nonconvexity arising from the combinatorial nature of the problem. Different modifications of the graph coloring problem and applications are also presented. The good performance of the method is shown in various computational experiments.
\end{abstract}


\section{Introduction}

A \emph{graph} $G=(V,E)$ is a collection of points $V$ that are connected by links $E\subset V\times V$. The points are usually known as  \emph{nodes} or \emph{vertices} while the links are called \emph{edges}, \emph{arcs} or \emph{lines}. An \emph{undirected graph} is a graph in which the edges have no orientation; that is, the edges are not ordered pairs of vertices but sets of two vertices.

A \emph{proper $m$-coloring} of an undirected graph $G$ is an assignment of one of $m$ possible colors to each vertex of $G$ such that no two adjacent vertices share the same color. More specifically, given the set of colors $K=\{1,\ldots,m\}$, an \emph{$m$-coloring} of G is a mapping $c:V\mapsto K$, assigning a color to each vertex. We say that $c$ is \emph{proper} if
$$
c(i)\neq c(j) \text{ for all } \{i,j\}\in E.
$$
The \emph{graph coloring problem} consists in determining whether is possible to find a proper $m$-coloring of the graph~$G$.
For a basic reference on graph coloring, see e.g.~\cite{JT95}.

Graph coloring has been used in many practical applications such as timetabling and scheduling~\cite{L79}, computer register allocation~\cite{CACCHM81}, radio frequency assignment~\cite{H80}, and printed circuit board testing~\cite{GJS76}. The graph coloring problem was proved to be NP-complete~\cite{K72}, so it is reasonable to believe that no polynomial-time exact algorithm solving these problems can be found. For this reason, a wide variety of heuristics and approximation algorithms have been developed for solving graph coloring problems. See~\cite{PMX98} for a bibliographic survey of algorithms and applications, or the more recent survey~\cite{FT12}.

In this paper we show that the Douglas--Rachford algorithm can be successfully used as  heuristic for solving a wide variety of graph coloring problems when they are conveniently modeled as feasibility problems. Despite the convergence of the Douglas--Rachford algorithm is only guaranteed for convex sets, the method has been successfully employed for solving many different nonconvex optimization problems, specially those of combinatorial nature (see, e.g.,~\cite{ABTmatrix, ABTcomb, BKroad, Elser}). The Douglas--Rachford method belongs to the family of so-called projection algorithms, which are traditionally analyzed using nonexpansivity properties when the problem is convex. There are very few results explaining why the algorithm works in nonconvex settings, and even less justifying its good global performance. For example, the global convergence of the algorithm for the case of a sphere and a line was proved in~\cite{benoist} (see also \cite{ABglobal, BS11}), and global convergence for the case of a halfspace and a potentially nonconvex set was recently proved in~\cite{ABT16}. For local convergence results involving nonconvex sets, see e.g.~\cite{BNlocal, HLnonconvex, Plinear}.

The good performance of the Douglas--Rachford algorithm for solving the problem consisting in coloring the edges of a complete graph with three colors while avoiding monochromatic triangles was shown by Elser et al. in~\cite{Elser}.  Elser seems to have been the first to see the remarkable potential of the algorithm for solving nonconvex problems.

The paper is structured as follows. Section~\ref{sec:Preliminaries} contains some preliminary results and notions. We show how to model the graph coloring problem as a feasibility problem in Section~\ref{sec:Modeling_graphcoloring}. When available, maximal clique information can be easily added to the model, as explained in Section~\ref{subsec:clique}. We present two ways of reformulating a 3-SAT problem as a graph coloring problem in Section~\ref{subsec:3sat}. The precoloring and list coloring problems, which are variations of the graph coloring problem, are discussed in Section~\ref{sec:precoloring}. We also treat in this section a well-known example of the precoloring problem: Sudoku puzzles. In Section~\ref{sec:8Q}, we show that the 8-queens puzzle, as well as some generalizations, can be also modeled as modified graph coloring problems. In Section~\ref{sec:hamiltonian}, we discuss the Hamiltonian path problem. We report the results of a collection of numerical experiments in Section~\ref{sec:numerical}, where we exhibit the good performance of the Douglas--Rachford method for finding a solution of all the graph coloring problems considered along the paper. We finish with various concluding remarks in Section~\ref{sec:conclusion}.

\section{Preliminaries}\label{sec:Preliminaries}

Let $\Hi$ be a Hilbert space with inner product $\langle\cdot ,\cdot\rangle$ and induced norm~\mbox{$\|\cdot\|$}. Given a nonempty subset $C\subseteq \Hi$ and $x\in \Hi$, a point $p\in C$ is said to be a \emph{best approximation} to $x$ from $C$ if
\begin{equation*}
\|p-x\|=d(x,C):=\inf_{c\in C}\|c-x\|.
\end{equation*}

If a best approximation in $C$ exists for every point in $\Hi$, then $C$ is called \emph{proximal}. The \emph{projector operator} onto C is the set-valued mapping $P_C:\Hi\rightrightarrows C$ given by
$$P_C(x):=\left\{p\in C :  \|p-x\|=\inf_{c\in C}\|c-x\|\right\},$$
and the \emph{reflector} is defined as $R_C:=2P_C-I$, where $I$ denotes the identity operator.
If every point $x\in \Hi$ has exactly one best approximation $p$, then $C$ is called \emph{Chebyshev} and $p$ is referred as the \emph{projection} of $x$ onto $C$. In this case, both $P_C$ and $R_C$ are single-valued. Recall that a weakly closed subset of a Hilbert space is convex if and only if it is a Chebyshev set (see, e.g.~\cite[Theorem~3.2]{ABMY14}).

Given $C_1,C_2,\ldots,C_r\subseteq \Hi$, the \emph{feasibility problem} consists in finding a point belonging to all these sets, that is,
\begin{equation*}\label{eq:FeasibilityProblem}
\text{Find } x\in\bigcap_{i=1}^r C_i.
\end{equation*}
In many practical situations, the projection onto each of these sets can be easily computed, while finding a point in the intersection of the sets might be intricate. In such cases, and when the sets are convex, the Douglas--Rachford method (DR in short) is a useful tool to solve the problem.

\begin{fact}
Let $A,B\subseteq\Hi$ be closed and convex sets. Consider the Douglas--Rachford operator defined as
$$T_{A,B}=\frac{I+R_BR_A}{2}.$$
Given any $x_0\in \Hi$, for every $n\geq0$, define $x_{n+1}=T_{A,B}(x_n)$. Then, the following holds.
\begin{itemize}[noitemsep,topsep=5pt]
\item [(i)] If $A\cap B\neq\emptyset$, then $\{x_n\}$ converges weakly to a point $x^\star$ such that $P_A(x^\star)\in A\cap B$.
\item [(ii)] If $A\cap B=\emptyset$, then $\|x_n\|\rightarrow +\infty$.
\end{itemize}
\end{fact}
\begin{proof}
See~\cite[Theorem~3.13 and Corollary~3.9]{BCL04}.
\end{proof}

Finitely many sets in a feasibility problem are usually handled by reducing the problem to the two-sets case trough the Pierra's \emph{product space formulation}. To this aim, consider the Hilbert product space $\Hi^r$ and define the sets
\begin{equation*}
C:=\prod_{i=1}^{r}C_i\quad\text{and} \quad D:=\left\{(x,x,\ldots,x)\in \Hi^r:x\in \Hi\right\}.
\end{equation*}
While the set $D$, sometimes called the \emph{diagonal}, is always a closed subspace, the properties of $C$ are largely inherited. For instance, $C$ is nonempty if $C_1,\ldots,C_r$ are not disjoint; and if $C_1,\ldots,C_r$ are closed and convex, so is $C$. Thus, the feasibility problem can be reformulated as a two-sets problem, since
\begin{equation*}
	x\in\bigcap_{i=1}^r C_i \Leftrightarrow (x,x,\ldots,x)\in C\cap D.
\end{equation*}
Moreover, knowing the projections onto $C_1,\ldots,C_r$, the projections onto $C$ and $D$ can be easily computed.
Indeed, for any $\mathbf{x}=(x_1,\ldots,x_r)\in \Hi^r$, we have
\begin{equation*}
P_C(\mathbf{x})=\prod_{i=1}^{r} P_{C_i}(x_i)\quad\text{and} \quad P_D(\mathbf{x})=\left( \frac{1}{r}\sum_{i=1}^r x_i \right)^r,
\end{equation*}
see~\cite[Lemma~1.1]{Pierra}. For further details see, for example,~\cite[Section~3]{ABTcomb}.

Throughout this paper the space $\Hi$ will be the Euclidean space $\R^{n\times m}$ of $n\times m$ real matrices. Its inner product is given by
$$\langle A,B\rangle := \tr\left(A^TB\right),$$
where $A^T$ is the transpose matrix of $A$, and $\tr(M)$ is the trace of  a square matrix $M$. The induced norm corresponds to the \emph{Frobenius} norm
$$\|A\|_F=\tr(A^TA)=\sqrt{\sum_{i=1}^n\sum_{j=1}^m a_{ij}^2}.$$

Let us introduce two results that characterize some projections on $\R^{n\times m}$, which will be useful later for computing the projection onto different sets.

\begin{fact}\label{fact:proj_unitvector}
Let $e_1,\ldots,e_n$ denote the unit vectors of the standard basis of $\R^n$, and consider $C=\{e_1,\ldots,e_n\}$. Then, for any $x=(x_1,\ldots,x_n)\in\R^n$,
$$P_C(x)=\left\{e_i: x_i=\max\left\{x_1,\ldots,x_n\right\}\right\}.$$
\end{fact}
\begin{proof}
See, e.g., \cite[Remark~5.1]{ABTcomb}.
\end{proof}

\begin{fact}\label{fact:proj_nullspace}
Let $A\in\R^{l\times n}$ be a full row rank matrix. Consider $C=\left\{Z\in\R^{n\times m}: AZ=0\right\}$. Then, for any $X\in\R^{n\times m}$, one has
$$P_{C}(X)=\left({\rm Id}_n-A^T\left(AA^T\right)^{-1}A\right)X,$$
where ${\rm Id}_n\in\R^{n\times n}$ denotes the identity matrix.
\end{fact}
\begin{proof}
See, e.g., \cite[Proposition~3.28(iii)]{BC11}.
\end{proof}

To finish this section, let us shortly summarize some basic concepts of graph theory. A \emph{complete graph} is an undirected graph in which every pair of nodes is connected by an edge. A \emph{clique} is a subset of vertices of an undirected graph such that its induced graph is complete. A \emph{maximal clique} is a clique that cannot be extended by adding one more vertex. A \emph{path} is a sequence of edges that connects a sequence of distinct vertices. A path is said to be a \emph{cycle} if there is an edge from the last vertex in the path to the first one.

\section{Modeling graph coloring problems as feasibility problems} \label{sec:Modeling_graphcoloring}

The $m$-coloring of a graph $G=(V,E)$ with $n$ nodes can be easily modeled as a feasibility problem. To this aim, let $X=(x_{ik})\in\{0,1\}^{n\times m}$, where $x_{ik}=1$ indicates that vertex $i$ receives color~$k$. Then, we have the following constraints:
\begin{gather}
\sum_{k=1}^m x_{ik}=1,\quad \text{for all } i=1,\ldots,n;\label{eq:C1}\\
x_{ik}+x_{jk}\leq 1, \quad \text{for all } \{i,j\}\in E, k=1,\ldots,m;\label{eq:C2old}\\
x_{ik}\in\{0,1\},\quad \text{for all } i=1,\ldots,n, k=1,\ldots,m.\label{eq:C3}
\end{gather}
Constraint~\eqref{eq:C1} together with~\eqref{eq:C3} determine that each node has exactly one color. Constraint~\eqref{eq:C2old} combined with~\eqref{eq:C3} impose the requirement that any two adjacent nodes cannot be assigned with the same color.

The formulation of the constraints has a big effect in the behavior of the Douglas--Rachford scheme when applied to nonconvex constraints. On one hand, ones needs a formulation where the projection onto the sets is easy to compute. On the other hand,  the formulation chosen often determines
whether or not the Douglas--Rachford scheme can successfully solve the problem at
hand always, frequently or never~\cite{ABTcomb}. For these two reasons, we have realized that it is convenient to reformulate constraint~\eqref{eq:C2old} as follows
\begin{equation}
x_{ik}+x_{jk}-y_{ek}=0, \quad \text{for all } e=\{i,j\}\in E, k=1,\ldots,m,\label{eq:C2_0}
\end{equation}
where $y_{ek}\in\{0,1\}$ for all $i,j\in\{1,\ldots,n\}$ and $k\in\{1,\ldots,m\}$.
Although we have considerably increased the number of variables of the feasibility problem by adding $lm$ new variables, we have empirically observed that the Douglas--Rachford scheme becomes much more successful with this formulation.

Finally, note that, since the labeling of the colors does not have any significant meaning, every permutation of a proper coloring is also a proper coloring. In our numerical tests we observed that this abundance of equivalent solutions significantly decreases the rate of success of the Douglas--Rachford algorithm. To avoid this problem, we restrict the set of possible colorings to those that assign the first color to the first vertex, that is, we add the constraint
\begin{equation}\label{eq:C5}
x_{1,1}=1.
\end{equation}
We shall also add the additional constraint that all $m$ colors have to be used, i.e.,
\begin{equation}\label{eq:C4}
\sum_{i=1}^n x_{ik}\geq 1,\quad\text{for all }k=1,\ldots,m.
\end{equation}

Let $E=\{e_1,\ldots, e_l\}$ be the set of edges, where $e_p\in\{1,\ldots,n\}^2$ for every $p=1,\ldots,l$. Let $I:=\{1,\ldots,n\}$ and $P:=\{n+1,\ldots,n+l\}$, and let $K:=\{1,\ldots,m\}$ be the set of colors. Then, the $m$-coloring problem determined by constraints~\eqref{eq:C1}, \eqref{eq:C3}, \eqref{eq:C2_0}, \eqref{eq:C5} and~\eqref{eq:C4} can be formulated as a feasibility problem with four constraints:
\begin{equation}\label{eq:formulation_1}
\text{Find }\quad Z\in C_1\cap C_2\cap C_3\cap C_4,
\end{equation}
where $Z=(z_{ik})\in\R^{(n+l)\times m}$ and
\begin{gather*}
C_1:=\left\{Z\in\R^{(n+l)\times m}:z_{ik}\in\{0,1\}, \forall (i,k)\in I\times K\text{ and } \sum_{k=1}^m z_{ik}=1, \forall i\in I\right\},\\
C_2:=\left\{Z\in\R^{(n+l)\times m}:z_{ik}+z_{jk}-z_{p k}=0, \text{with }e_{p-n}=\{i,j\} \in E,\forall (p,k)\in P\times K\right\},\\
C_3:=\left\{Z\in\{0,1\}^{(n+l)\times m}:\sum_{i=1}^n z_{ik}\geq 1, \forall k\in K\right\},\\
C_4:=\left\{Z\in\R^{(n+l)\times m}:z_{1,1}=1\right\}.
\end{gather*}
Observe that constraint $C_2$ can be expressed in matrix form as
\begin{equation}\label{eq:C2}
C_2=\left\{Z\in\R^{(n+l)\times m}: AZ=0_{l\times m}\right\},
\end{equation}
where $A=(a_{pq})\in\R^{l\times(n+l)}$ is defined by
$$a_{pq}:=\left\{\begin{array}{rl}1 & \text{if } e_p=\{i,j\}\text{ and }q\in\{i,j\},\\
-1 & \text{if } q=n+p,\\
0 & \text{elsewhere;}
\end{array}\right.$$
for each $p=1,\ldots,l$ and $q\in I\cup P$.

The projections onto each of the above sets can be derived from Fact~\ref{fact:proj_unitvector} and Fact~\ref{fact:proj_nullspace}. The projections of any $Z\in\R^{(n+l)\times m}$ onto $C_1$, $C_3$ and $C_4$ are given, pointwise, by
\begin{gather*}
\left(P_{C_1}(Z)\right)[i,k]=\left\{\begin{array}{ll} 1&\text{if } i\in I,k=\arg\!\max\{z_{i1},z_{i2},\ldots,z_{im}\},\\
        z_{ik}&\text{if } i\in P,\\
                                      0&\text{otherwise};\end{array}\right.\\
\left(P_{C_3}(Z)\right)[i,k]=\left\{\begin{array}{ll} 1&\text{if } i=\arg\!\max\{z_{1k},z_{2k},\ldots,z_{nk}\},\\
        \min\left\{1,\max\left\{0,\round(z_{ik})\right\}\right\}&\text{otherwise};\end{array}\right.\\
\left(P_{C_4}(Z)\right)[i,k]=\left\{\begin{array}{ll} 1&\text{if } i=k=1,\\
z_{ik} &\text{otherwise};\end{array}\right.
\end{gather*}
for each $i\in I\cup P$ and $k\in K$, where the lowest index is chosen in $\arg\!\max$ (the projections onto $C_1$ and $C_3$ may not be unique). Since $A$ is full row rank, the projection onto $C_2$ is given by $$P_{C_2}(Z)=\left({\rm Id}_{n+l}-A^T\left(AA^T\right)^{-1}A\right)Z.$$

\subsection{Adding maximal clique information}\label{subsec:clique}

Let us consider now the so-called \emph{windmill graph} $\Wd(a,b)$, which is the graph constructed for $a\geq 2$ and $b\geq 2$ by joining $b$ copies of a complete graph with $a$ vertices at a shared vertex. A plot of $\Wd(6,4)$ is shown in Figure~\ref{fig:windmill}.

\begin{figure}[ht!]
\centering
\begin{tikzpicture}[scale=.61]%
    \GraphInit[vstyle=Normal]
    \SetVertexNoLabel
    \tikzset{VertexStyle/.style = {draw,shape = circle,minimum size = 14pt,inner sep=0pt}}
    \grComplete[x=-2.9,y=4,rotation=90,prefix=a,form=1,RA=2]{5};
    \grComplete[x=2.9,y=4,rotation=90,prefix=b,form=1,RA=2]{5};
    \grComplete[x=3,y=-4,rotation=55,prefix=d,form=1,RA=2]{5};
    \grComplete[x=-2.8,y=-4,rotation=55,prefix=e,form=1,RA=2]{5};
    \Vertex[x=0, y=0] {c0}{1};
    \EdgeFromOneToAll{c}{a}{0}{5};
    \EdgeFromOneToAll{c}{b}{0}{5};
    \EdgeFromOneToAll{c}{d}{0}{5};
    \EdgeFromOneToAll{c}{e}{0}{5};
    \AssignVertexLabel{a}{2,...,6};
    \AssignVertexLabel{b}{7,...,11};
    \AssignVertexLabel{c}{1};
    \AssignVertexLabel{d}{12,...,16};
    \AssignVertexLabel{e}{17,...,21};
\end{tikzpicture}
\caption{Plot of the windmill graph $\Wd(6,4)$.}\label{fig:windmill}
\end{figure}
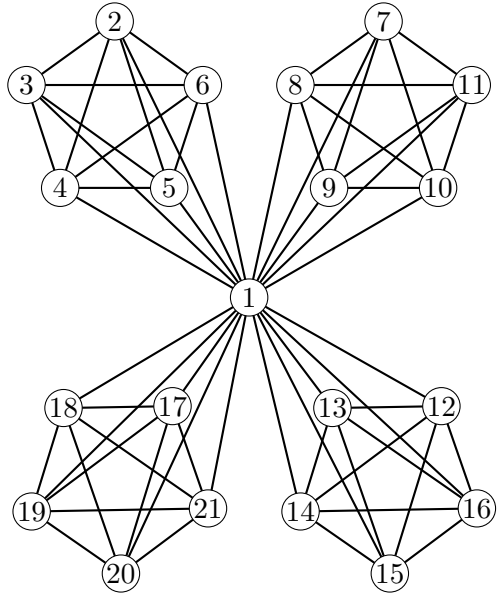

Every windmill graph $\Wd(a,b)$ can be easily $a$-colored (there are $a((a-1)!)^b$ different ways). Despite this abundance of valid colorings, the Douglas--Rachford scheme described in the previous section fails to find a solution rather often, see the results in Figure~\ref{fig:Exp_WindmillCliques}. This graph has an additional available information that can be used: it has $b$ maximal cliques of length~$a$, and each color can be used at most once within each maximal clique.

Let $Q\subset 2^V$ be a nonempty subset of maximal cliques of the graph $G=(V,E)$ and let $\widehat E:=E\cup Q$. Let $Q=\{e_{l+1},\ldots,e_{r}\}$, with $r\geq l+1$. Thus, $\widehat E=\{e_1,\ldots,e_l,e_{l+1},\ldots,e_{r}\}$. The maximal clique information can be easily added into constraint $C_2$ in~\eqref{eq:C2}. Indeed, let
\begin{equation*}
\widehat C_2:=\left\{Z\in\R^{\left(n+r\right)\times m}: \widehat AZ=0_{r\times m}\right\},
\end{equation*}
where $\widehat A=(\widehat a_{pq})\in\R^{r\times(n+r)}$ is defined by
$$\widehat a_{pq}:=\left\{\begin{array}{rl}1 & \text{if } q\in e_p,\\
-1 & \text{if } q=n+p,\\
0 & \text{elsewhere;}
\end{array}\right.$$
for each $p=1,\ldots,r$ and $q\in \{1,\ldots,n+r\}$. This is clearly an equivalent formulation of the $m$-coloring problem, where we have added $(r-l)m$ new variables (now $Z\in \R^{\left(n+r\right)\times m}$), which correspond to the (redundant) information that each color can only be used once within each maximal clique. Despite that, this formulation can be advantageous, as shown in Figure~\ref{fig:Exp_WindmillCliques}. For some particular graphs, adding this information can be crucial, see Table~\ref{tbl:3-SAT}, where we compare two reformulations of 3-SAT problems with and without maximal clique information. Explaining these reformulations is the subject of the next section.

\subsection{Formulating 3-SAT as 3-coloring}\label{subsec:3sat}

A \emph{Boolean variable} takes logical values:  True (T) or False (F). A \emph{literal} is either a variable or its negation ($\lnot$). A \emph{clause} is a disjunction ($\lor$) of literals. A formula in  \emph{conjunctive normal form} is a conjunction ($\land$) of clauses.
Given a formula in conjunctive normal form with $3$ literals per clause, the \emph{3-SAT} (\emph{3-satisfiability}) problem consists in determining if there exists an assignment of variables that makes the formula true. Specifically, let $x_1, \ldots, x_n$ be $n$ Boolean variables and consider $m$ clauses $\theta_1, \ldots, \theta_m$, where each clause is the disjunction of $3$ literals, $$\theta_j=t_1^j\lor t_2^j \lor t_3^j, \quad \text{for all } j=1, 2, \ldots, m;$$
with $t_1^j,t_2^j,t_3^j\in\bigcup_{i=1}^n\{x_i,\lnot x_i\}$. Let $\phi$ be the formula comprising the conjunction of all the clauses:
$$\phi=\theta_1\land \theta_2\land \cdots \land\theta_m.$$
Then, the 3-SAT problem consists in determining if there exists an assignment of the variables that makes the formula $\phi$ true.

 \begin{example}\label{exa:3sat}
 	Consider the following 3-SAT problem with $3$ variables and $2$ clauses:
 	$$\phi=\left( x_1 \lor x_2 \lor x_3 \right)\land \left( \lnot x_1 \lor x_2 \lor \lnot x_3 \right).$$
 	There are several solutions to $\phi$ such as $(F,T,F)$, $(T,T,F)$ or $(F,F,T)$, among others.
 \end{example}

A 3-SAT problem can be reduced to a 3-coloring problem by using gadgets. A \emph{gadget}
is a small graph whose coloring solves some part of the problem. Using a set of gadgets and connecting them in an appropriate manner, the 3-coloring problem of the full graph can be made equivalent to solving the 3-SAT problem. We start by creating $n+1$ gadgets, one for each variable and an additional one for setting the interpretation of the colors:

 \begin{itemize}
 	\item[(a)] Create a gadget formed by a complete graph with 3 ``color-meaning'' nodes named T, F and G, see Figure~\ref{fig:gadgetcommon}(a). As this gadget is a complete graph, a different color must be assigned to each node. The color assigned to node T will be interpreted as True, the color assigned to F as False, and the remaining color assigned to G (\emph{ground} node) will not have any special interpretation.
 	
 	\item[(b)] For each variable $x_i$, construct a gadget with 2 connected nodes, one associated to $x_i$ and the other to $\lnot x_i$. Link both of them to the node G to create a gadget of the form in Figure~\ref{fig:gadgetcommon}(b). This gadget forces a logical choice in the value of the variables. Thus, every variable will be assigned to either T or F, and the assignment of every variable and its complement will be consistent.
 \end{itemize}

 \begin{figure}[ht!]
 	\centering
 	\subfigure[]{
 		\begin{tikzpicture}[scale=.6,transform shape]%
 		\SetVertexNoLabel
 		\begin{scope}[VertexStyle/.append style = {minimum size = 25pt, draw=black}]
 		\Vertex[x=-2.5, y=4] {f0}
 		\Vertex[x=0, y=0] {g0}
 		\Vertex[x=2.5, y=4] {t0}
 		\AssignVertexLabel{f}{\Large F}
 		\AssignVertexLabel{g}{\Large G}
 		\AssignVertexLabel{t}{\Large T}
 		\end{scope}
 		
 		\Edges(f0,g0,t0,f0)
 		\end{tikzpicture}}%
 	\qquad%
 	\subfigure[]{
 		\begin{tikzpicture}[scale=.6,transform shape]%
 		\GraphInit[vstyle=Normal]
 		\SetVertexNoLabel
 		\begin{scope}[VertexStyle/.append style = {minimum size = 25pt, draw=black}]
 		\Vertex[x=0, y=4] {g0}
 		\Vertex[x=-1.5, y=0] {t0}
 		\Vertex[x=1.5, y=0] {t1}
 		\AssignVertexLabel{t}{\Large $x_i$,\Large $\lnot{x_i}$}
 		\AssignVertexLabel{g}{\Large G}
 		\end{scope}
 		
 		\Edges(t0,t1)
 		\tikzstyle{EdgeStyle}=[dashed]
 		\Edges(t0,g0,t1)
 		\end{tikzpicture}}%
 	\caption{Gadgets of the variables and colors.}\label{fig:gadgetcommon}
 \end{figure}
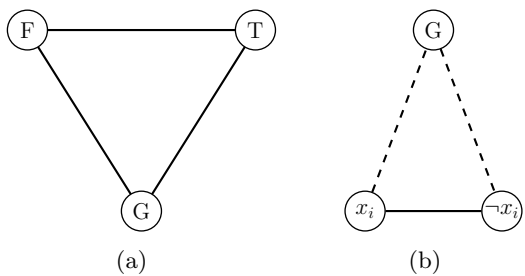

Next, we present two different formulations of the gadgets corresponding to the clauses.

 \begin{itemize}
    \item [(c)] For the \emph{4-nodes formulation}, take each clause $\theta=t_1\lor t_2 \lor t_3$ and create the gadget in Figure~\ref{fig:gadget2}(a) with the nodes associated to $t_1$, $t_2$, $t_3$, F, G, and 4 new nodes. The new unlabeled nodes do not have any special meaning, but, by the construction of the gadgets, every 3-coloring of a clause gadget will assign the same color as T to at least one of the literals $t_1$, $t_2$ or $t_3$. Thus, a valid 3-coloring of the gadget will make the corresponding clause to be True.

        For the \emph{5-nodes formulation}, the process is similar but introduces five new nodes instead of four: the gadget is shown in Figure~\ref{fig:gadget2}(b).
\end{itemize}
 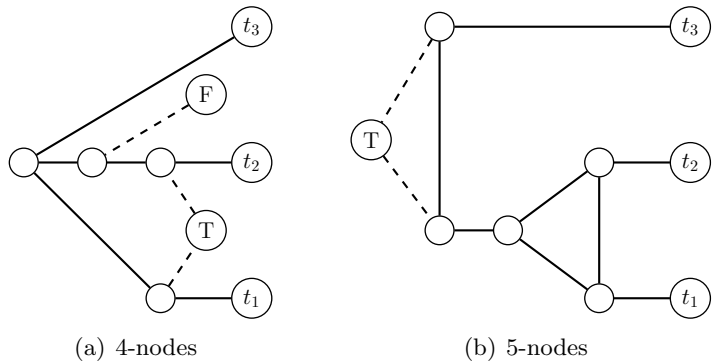
\begin{figure}[ht!]
 	\centering
 	\subfigure[4-nodes]{
 		\begin{tikzpicture}[scale=.6,transform shape]%
 		\GraphInit[vstyle=Normal]
 		\SetVertexNoLabel
 		\begin{scope}[VertexStyle/.append style = {minimum size = 25pt, draw=black}]
 		\Vertex[x=-1, y=0] {t0}
 		\Vertex[x=-1, y=3] {t1}
 		\Vertex[x=-1, y=6] {t2}
 		\AssignVertexLabel{t}{\Large $t_1$,\Large $t_2$,\Large $t_3$}
 		\Vertex[x=-2, y=1.5] {tr0}
 		\AssignVertexLabel{tr}{\Large T}
 		\Vertex[x=-2, y=4.5] {f0}
 		\AssignVertexLabel{f}{\Large F}
 		\end{scope}
 		
 		\Vertex[x=-3, y=0] {c0}
 		\Vertex[x=-3, y=3] {c1}
 		\Vertex[x=-4.5, y=3] {c2}
 		\Vertex[x=-6, y=3] {c3}
 		
 		\Edges(t0,c0,c3,t2)
 		\Edges(t1,c1,c2,c3)
 		\tikzstyle{EdgeStyle}=[dashed]
 		\Edges(c0,tr0,c1)
 		\Edges(f0,c2)	
 		\end{tikzpicture}}
 	\qquad%
 	\subfigure[5-nodes]{
 		\begin{tikzpicture}[scale=.6,transform shape]%
 		\GraphInit[vstyle=Normal]
 		\SetVertexNoLabel
 		\begin{scope}[VertexStyle/.append style = {minimum size = 25pt, draw=black}]
 		\Vertex[x=9, y=0] {t0}
 		\Vertex[x=9, y=3] {t1}
 		\Vertex[x=9, y=6] {t2}
 		\AssignVertexLabel{t}{\Large $t_1$,\Large $t_2$,\Large $t_3$}
 		\Vertex[x=2, y=3.5] {m0}
 		\AssignVertexLabel{m}{\Large T}
 		
 		\end{scope}
 		
 		\Vertex[x=7, y=0] {c0}
 		\Vertex[x=7, y=3] {c1}
 		\Vertex[x=3.5, y=6] {c2}
 		\Vertex[x=5, y=1.5] {e0}
 		\Vertex[x=3.5, y=1.5] {e1}
 		
 		\Edges(t0,c0,e0,e1)
 		\Edges(t1,c1)
 		\Edges(t2,c2,e1)
 		\Edges(c0,c1,e0)
 		\tikzstyle{EdgeStyle}=[dashed]
 		\Edges(c2,m0,e1)
 		\end{tikzpicture}}
 	\caption{Gadgets of the clauses.}\label{fig:gadget2}
 \end{figure}
 \begin{itemize}
    \item[(d)] Finish building the graph by connecting the clause gadgets together using the edges from the common gadgets from Figure~\ref{fig:gadgetcommon}. Full graphs for the four and five node formulations of the 3-SAT problem in Example~\ref{exa:3sat} are shown in Figure~\ref{fig:graph_gadget}.
 \end{itemize}

The graph resulting from putting all these gadgets together in the 4-nodes formulation has a total of $3+2n+4m$ nodes and $3+3n+9m$ edges. Observe that the graph has $n+1$ maximal cliques with $3$~nodes, one for each gadget of type~(a) and~(b). In the 5-nodes formulation, the resulting graph has a total of $3+2n+5m$ nodes and $3+3n+10m$ edges. The number of maximal cliques with 3 nodes has increased up to $n+1+2m$, one for each gadget of type~(a) and~(b) and two for each gadget of type~(c). A 3-coloring of the graph built under one of these two formulations corresponds to a solution of the associated 3-SAT problem. A solution to the 3-SAT problem in Example~\ref{exa:3sat}  using both formulations is shown in Figure~\ref{fig:graph_gadget}.

 \begin{figure}[ht!]
 \centering
 \subfigure[4-nodes formulation.]{
 	\centering
 	\begin{tikzpicture}[scale=.69,transform shape]%
 	\GraphInit[vstyle=Normal]
 	\SetVertexNoLabel
 	\begin{scope}[VertexStyle/.append style = {minimum size = 25pt, color=green!50, draw=black}]
 	\Vertex[x=-1, y=1] {t0}
 	\Vertex[x=1.2, y=3] {tt1}
 	\Vertex[x=1.4, y=5] {tt2}
 	\Vertex[x=2, y=8] {tr0}
 	\AssignVertexLabel{tr}{\Large T}
 	\end{scope}
 	
 	\begin{scope}[VertexStyle/.append style = {minimum size = 25pt, color=red!50, draw=black}]
 	\Vertex[x=-1.2, y=3] {t1}
 	\Vertex[x=-1.4, y=5] {t2}
 	\Vertex[x=1, y=1] {tt0}
 	\Vertex[x=-0.5, y=8.9] {f0}
 	\AssignVertexLabel{f}{\Large F}
 	\end{scope}
 	
 	\begin{scope}[VertexStyle/.append style = {minimum size = 25pt, color=blue!60, draw=black}]
 	\Vertex[x=0, y=6.5] {g0}
 	\AssignVertexLabel{g}{\Large G}
 	\end{scope}
 	
 	\AssignVertexLabel{t}{\Large $x_1$,\Large ${x_2}$,\Large $x_3$}
 	\AssignVertexLabel{tt}{\Large $\lnot{x_1}$,\Large $\lnot{x_2}$,\Large $\lnot{x_3}$}
 	\tikzstyle{VertexStyle}=[{circle, draw=black, minimum size=15pt}]
 	\Vertex[x=-3, y=1] {c0}
 	\Vertex[x=-3, y=3] {c1}
 	\Vertex[x=-4.5, y=3] {c2}
 	\Vertex[x=-6, y=3] {c3}
 	\Vertex[x=3, y=1] {cc0}
 	\Vertex[x=3, y=3] {cc1}
 	\Vertex[x=4.5, y=3] {cc2}
 	\Vertex[x=6, y=3] {cc3}
 	
 	\AddVertexColor{green!50}{cc2,c2}
 	\AddVertexColor{red!50}{cc3,c0}
 	\AddVertexColor{blue!60}{cc0,cc1,c1,c3}
 	
 	\Edges(t0,c0,c3,t2)
 	\Edges(t1,c1,c2,c3)
 	\Edges(tt0,cc0,cc3,tt2)
 	\Edges(cc1,cc2,cc3)
 	\Edges(g0,f0,tr0,g0)
 	
 	\tikzstyle{EdgeStyle}=[bend right=20]
 	\Edges(t0,tt0)
 	\Edges(t1,tt1)
 	\Edges(t2,tt2)
 	
 	\tikzstyle{EdgeStyle}=[bend right=30]
 	\Edges(cc1,t1)
 	
 	\tikzstyle{EdgeStyle}=[dashed, bend right=60]
 	\Edges(tr0,c0)
 	\tikzstyle{EdgeStyle}=[dashed, bend right=40]
 	\Edges(tr0,c1)
 	\tikzstyle{EdgeStyle}=[dashed, bend right=40]
 	\Edges(f0,c2)
 	\tikzstyle{EdgeStyle}=[dashed, bend right=30]
 	\Edges(cc0,tr0)
 	\tikzstyle{EdgeStyle}=[dashed, bend right=20]
 	\Edges(cc1,tr0)
 	\tikzstyle{EdgeStyle}=[dashed, bend right=65]
 	\Edges(cc2,f0)	
 	\tikzstyle{EdgeStyle}=[dashed]
 	\Edges(t0,g0)
 	\Edges(t1,g0)
 	\Edges(t2,g0)
 	\Edges(tt0,g0)
 	\Edges(tt1,g0)
 	\Edges(tt2,g0)	
 	\end{tikzpicture}
 }\\
\subfigure[5-nodes formulation.]{
	\centering
	\begin{tikzpicture}[scale=0.69,transform shape]%
	\GraphInit[vstyle=Normal]
	\SetVertexNoLabel
	\begin{scope}[VertexStyle/.append style = {minimum size = 25pt, color=green!50, draw=black}]
	\Vertex[x=7, y=8] {tr0}
	\AssignVertexLabel{tr}{\Large T}
	\Vertex[x=5.4, y=0] {t0}
	\Vertex[x=7, y=2] {tt1}
	\Vertex[x=7.2, y=4] {tt2}
	\end{scope}
	\begin{scope}[VertexStyle/.append style = {minimum size = 25pt, color=red!50, draw=black}]
	\Vertex[x=8, y=6] {f0}
	\AssignVertexLabel{f}{\Large F}
	\Vertex[x=6.6, y=0] {tt0}
	\Vertex[x=5, y=2] {t1}
	\Vertex[x=4.8, y=4] {t2}
	\end{scope}
	\begin{scope}[VertexStyle/.append style = {minimum size = 25pt, color=blue!60, draw=black}]
	\Vertex[x=6, y=6] {g0}	
	\AssignVertexLabel{g}{\Large G}
	\end{scope}	
	
	\AssignVertexLabel{t}{\Large $x_1$,\Large ${x_2}$,\Large $x_3$}
	\AssignVertexLabel{tt}{\Large $\lnot{x_1}$,\Large $\lnot{x_2}$,\Large $\lnot{x_3}$}
	
	\tikzstyle{VertexStyle}=[{circle, draw=black, minimum size=15pt}]
	
	\Vertex[x=3, y=0] {c0}
	\Vertex[x=3, y=2] {c1}
	\Vertex[x=1, y=4] {c2}
	\Vertex[x=2, y=1] {e0}
	\Vertex[x=1, y=1] {e1}
	\Vertex[x=9, y=0] {cc0}
	\Vertex[x=9, y=2] {cc1}
	\Vertex[x=11, y=4] {cc2}
	\Vertex[x=10, y=1] {ee0}
	\Vertex[x=11, y=1] {ee1}
	
	\AddVertexColor{green!50}{e0,cc0}
	\AddVertexColor{red!50}{e1,c0,ee0,cc2}
	\AddVertexColor{blue!60}{c2,c1,cc1,ee1}
	
	\Edges(t0,c0,e0,e1)
	\Edges(t1,c1)
	\Edges(t2,c2,e1)
	\Edges(c0,c1,e0)
	\Edges(tt0,cc0,ee0,ee1)
	\Edges(tt2,cc2,ee1)
	\Edges(cc0,cc1,ee0)	
	\Edges(f0,g0,tr0,f0)
	\tikzstyle{EdgeStyle}=[bend right=20]
	\Edges(t0,tt0)
	\Edges(t1,tt1)
	\Edges(t2,tt2)
	\tikzstyle{EdgeStyle}=[bend left=30]
	\Edges(t1,cc1)
	\tikzstyle{EdgeStyle}=[dashed]
	\Edges(t0,g0)
	\Edges(t1,g0)
	\Edges(t2,g0)
	\Edges(tt0,g0)
	\Edges(tt1,g0)
	\Edges(tt2,g0)	
	\tikzstyle{EdgeStyle}=[bend left=30, dashed=4]
	\Edges(e1,tr0)
	\Edges(c2,tr0)
	\tikzstyle{EdgeStyle}=[bend left=25, dashed=4]
	\Edges(tr0,ee1)
	\tikzstyle{EdgeStyle}=[bend left=30, dashed=4]
	\Edges(tr0,cc2)
	\end{tikzpicture}
}
 	\caption{Two different formulations of the 3-SAT problem in Example~\ref{exa:3sat} as a 3-coloring problem. The same solution of the 3-SAT problem is shown for both formulations.}\label{fig:graph_gadget}
 \end{figure}
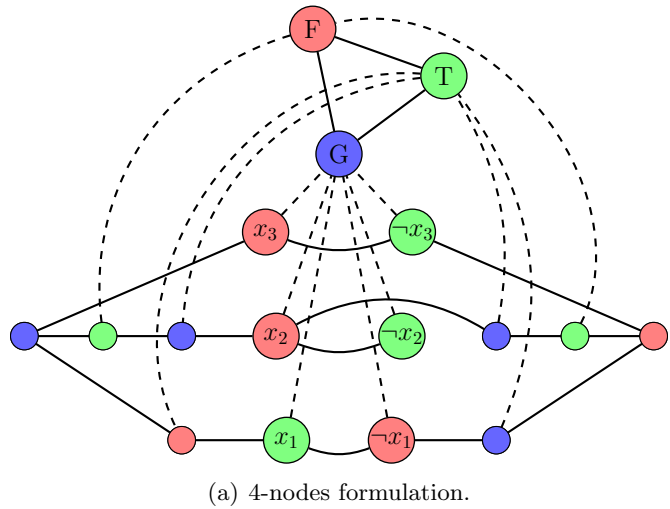
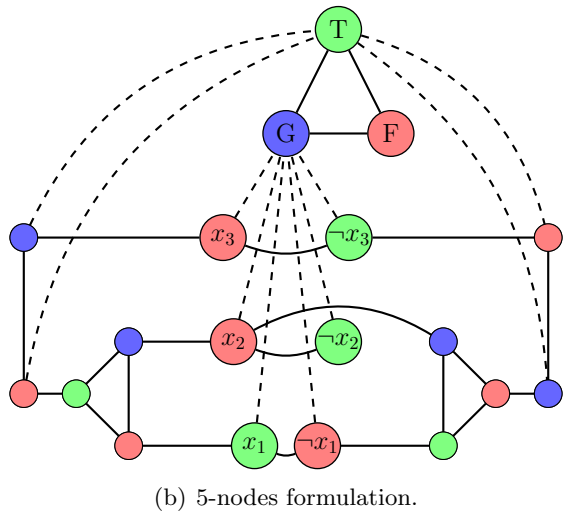

The results of testing the performance of the Douglas--Rachford method for solving a sample of 3-SAT problems using both reformulations as 3-coloring problems is shown in Section~\ref{sec:numerical}, see Table~\ref{tbl:3-SAT}. With a totally different direct formulation, the Douglas--Rachford method was first shown to be successful for solving 3-SAT problems in~\cite{Elser}.

\section{Precoloring and list coloring problems}\label{sec:precoloring}

In many practical graph coloring problems, the set of eligible colors for each of the nodes can be different. This is the case in the \emph{precoloring problem}, a slight modification of the graph coloring problem in which a subset of the vertices has been preassigned to some colors. The task is to color the remaining vertices to obtain a valid coloring of the entire graph. More generally, in the \emph{list coloring problem}, each vertex can only be colored from a list of admissible colors.

The notion of list coloring was  independently introduced by Vizing~\cite{V76}, and Erd\"{o}s, Rubin and Taylor~\cite{ERT79}. Given a graph $G=(V,E)$ and a set of $m$ colors $K=\{1,\ldots,m\}$, let $L:V\rightrightarrows K$ be a mapping assigning to each vertex $v\in V$ a list of admissible colors $L(v)\subseteq K$. Thus, the list coloring problem consists in finding a proper coloring of the vertices of the graph $G$ verifying that the color assigned to each vertex belongs to its list of admissible colors; that is, $c(i)\neq c(j) \text{ for all } \{i,j\}\in E, \text{ and } c(i)\in L(i) \text { for all } i\in V$. Note that an ordinary graph coloring problem is a special case of list coloring where $L(i)=K$ for every vertex $i\in V$, and so are the precoloring problems, where the precolored vertices have a list of admissible colors with length one.

List coloring problems can be reduced to standard graph coloring problems. To this aim, one shall add a complete subgraph with $m$ new nodes, each one representing a color in $K$, and connect each vertex $i\in V$ with the new nodes that represent the colors not belonging to~$L(i)$. If we denote by $|A|$ the cardinality of a finite set $A$, the new graph will have $n+m$ nodes, $l^\star=|E|+\frac{m(m-1)}{2}+nm-\sum_{i=1}^n |L(i)|$ edges, and an additional maximal clique of length $m$. In this way, any valid $m$-coloring of the extended graph will lead to a solution for the original list coloring problem. An example of such construction with a wheel graph of $5$ nodes is shown in Figure~\ref{fig:listcoloring}.

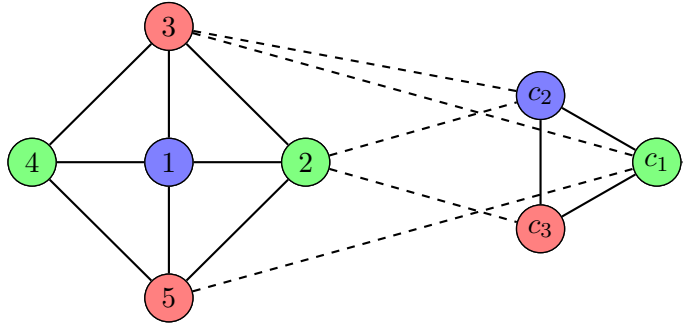
\begin{figure}[ht!]
	\centering
	\begin{tikzpicture}[scale=.6]%
	\GraphInit[vstyle=Normal]
	\SetVertexNoLabel
	\tikzstyle{EdgeStyle}=[thick]
	\grWheel[rotation=0,prefix=b,RA=3]{5};
	
	\begin{scope}[xshift=9cm]
	\tikzstyle{EdgeStyle}=[thick]
	\grComplete[rotation=0,prefix=c,RA=1.7]{3}
	
	\end{scope}
	
	\AddVertexColor{green!50}{c0,b0,b2}
	\AddVertexColor{blue!50}{c1,b4}
	\AddVertexColor{red!50}{c2,b1,b3}
	\AssignVertexLabel{c}{$c_1$,$c_2$,$c_3$}
	\AssignVertexLabel{b}{2,...,5,1};
	
 	\tikzstyle{EdgeStyle}=[dashed]
 	\Edges(b0,c1)
 	\Edges(b1,c0)
 	\Edges(b3,c0)
 	\Edges(b1,c1)
 	\Edges(b0,c2)	
	\end{tikzpicture}
	\caption{List coloring reduced to graph coloring of a wheel graph of 5 nodes with admissible colors lists $L(1)=L(4)=\{1,2,3\}, L(2)=\{1\}, L(3)=\{3\}$, and $L(5)=\{2,3\}$. Nodes $c_1$, $c_2$ and $c_3$ represent colors 1, 2 and 3, respectively.}\label{fig:listcoloring}
\end{figure}

Note that the new feasibility problem is defined in $\R^{(n+m+l^\star)\times m}$. Constraint $C_4$ has to be changed, as it no longer makes sense. We have $m$ new nodes, labeled $n+1,\ldots,n+m$, each of them representing a color. To include this information, we shall replace $C_4$ by
$$
C_4^\star:=\left\{Z\in\R^{(n+m+l^\star)\times m}:z_{n+k,k}=1, \forall k\in K\right\}.
$$
Thereby, the solution set is $C_1\cap C_2\cap C_3\cap C_4^\star$. The projection onto $C_4^\star$ is given by
$$
\left(P_{C_4^\star}(Z)\right)[i,k]=\left\{\begin{array}{ll} 1&\text{if } i=k+n,\\
z_{ik} &\text{otherwise}.\end{array}\right.
$$

As the increase in the number of nodes and edges may cause the DR algorithm to become slower, another option here would be to directly modify the constraint~$C_1$ to only allow admissible colors, that is, to replace it by the set
$$
\overline{C}_1:=\left\{Z\in\R^{(n+l)\times m}:z_{ik}\in\{0,1\}, \forall (i,k)\in I\times K\text{ and } \sum_{k\in L(i)} z_{ik}=1, \forall i\in I\right\},
$$
whose projection is given by
$$
\left(P_{\overline{C}_1}(Z)\right)[i,k]=\left\{\begin{array}{ll} 1&\text{if } i\in I,k=\arg\!\max\{z_{ij}, j\in L(i)\},\\
z_{ik}&\text{if } i\in P,\\
0&\text{otherwise}.\end{array}\right.
$$
Constraint $C_4$ has to be removed from the feasibility problem, and the solution set becomes $\overline{C}_1\cap C_2\cap C_3$. We shall compare the performance of DR with both formulations in Section~\ref{sec:numerical}.	

\subsection{Formulating Sudokus as 9-precoloring problems}\label{subsec:Sudokus}

It is easy to formulate \emph{Sudoku puzzles} as graph coloring problems. This kind of puzzles consist in a  $9\times9$ grid, divided in nine $3\times3$ subgrids, with some entries already prefilled. The objective is to fill the remaining cells in such a way that each row, each column and each subgrid contains the digits from $1$ to $9$ exactly once.

We shall model Sudokus as 9-precoloring problems, with the aim of applying DR. The construction of the graph is very simple and intuitive. Each cell in the grid shall be represented by a node. Then, we link two nodes if their respective associated cells lay in the same row, same column or same subgrid (see Figure~\ref{fig:Sudoku_graph}). The graph obtained contains $81$ nodes and $810$ edges. Furthermore, a rich maximal clique information is known. Namely, there are $27$ maximal cliques of size $9$, one per row, one per column and one per subgrid.

\begin{figure}[ht!]
	\Large
	\centering
	\begin{tikzpicture}[scale=.5,transform shape, darkstyle/.style={circle,draw,fill=white,minimum size=22}]
	\def\sepa{2}
	\def\sopa{1.1}
	\def\ma{0.6}
	\def\ra{3.45}
	
	\foreach \x in {0,1,2}
	\foreach \u in {0,1,2}
	{\pgfmathtruncatemacro{\xx}{int(3*\x+\u)}
		\draw[fill=blue!50, draw=blue!50!black, thick, dotted, fill opacity=0.2] (-3*\ma,-\sepa*\xx-\sopa*\x+\ma) rectangle (8*\sepa+2*\sopa+3*\ma,-\sepa*\xx-\sopa*\x-\ma);
		\draw[fill=green!50, draw=green!50!black, thick, dashed, fill opacity=0.2] (\sepa*\xx+\sopa*\x-\ma,3*\ma) rectangle (\sepa*\xx+\sopa*\x+\ma,-8*\sepa-2*\sopa-3*\ma);}
	
	\foreach \x in {0,1,2}
	\foreach \y in {0,1,2}
	{\pgfmathtruncatemacro{\xx}{int(3*\x+1)}
		\pgfmathtruncatemacro{\yy}{int(3*\y+1)}
		\draw[fill=red!50, draw=red!40!black, fill opacity=0.2] (\sepa*\yy+\sopa*\y,-\sepa*\xx-\sopa*\x) circle [radius=\ra cm];}
	
	\foreach \x in {0,1,2}
	\foreach \u in {0,1,2}
	\foreach \y in {0,1,2}
	\foreach \v in {0,1,2}
	{\pgfmathtruncatemacro{\xx}{int(3*\x+\u)}
		\pgfmathtruncatemacro{\yy}{int(3*\y+\v)}
		\pgfmathtruncatemacro{\label}{\xx * 9 +  \yy + 1}
		\node [darkstyle]  (\xx\yy) at (\sepa*\yy+\sopa*\y,-\sepa*\xx-\sopa*\x) {\label};}	
	\end{tikzpicture}
	\caption{Graph formulation of a Sudoku, with maximal cliques highlighted.}\label{fig:Sudoku_graph}
\end{figure}
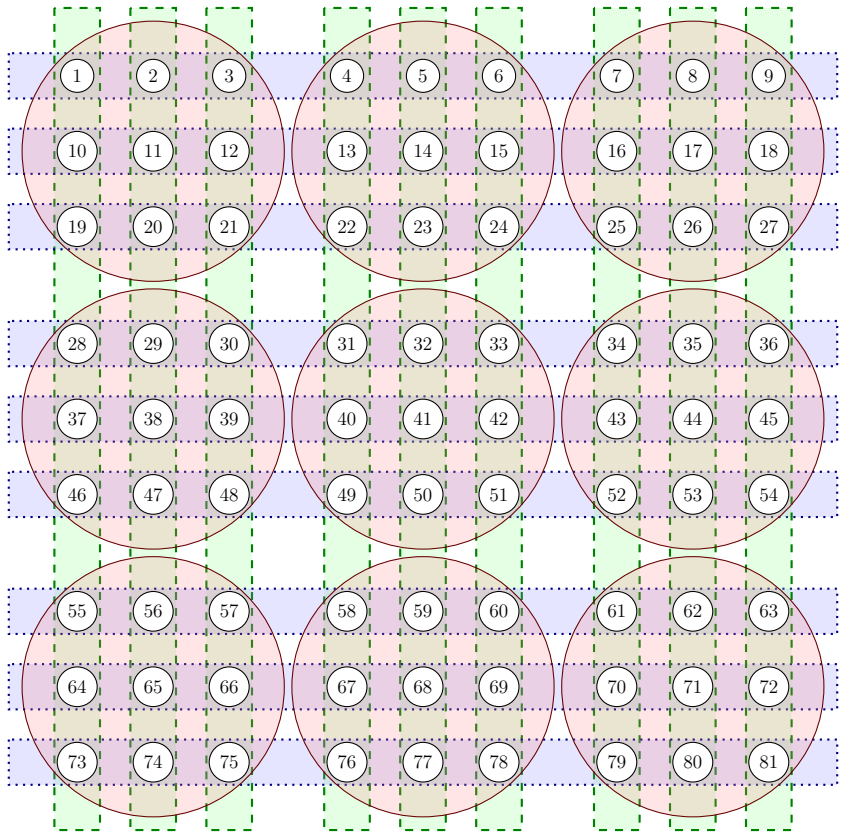

Sudoku puzzles can be directly modeled as integer feasibility programs. Despite the Douglas--Rachford algorithm fails to solve these integer problems, it can be successfully used for solving the puzzles after reformulating them as binary programs, see~\cite[Section~6]{ABTcomb}. We must acknowledge here the fundamental contribution of Veit Elser~\cite{Elser}, who first realized the usefulness of this binary reformulation for the success of the DR algorithm.

We associate a color to each of the $9$ digits of the puzzle. Since some cells of the Sudoku are prefilled, this is actually a graph precoloring problem. A valid coloring of the graph will lead to a solution of the Sudoku, as shown in the example in Figure~\ref{fig:Sudoku_colored}.

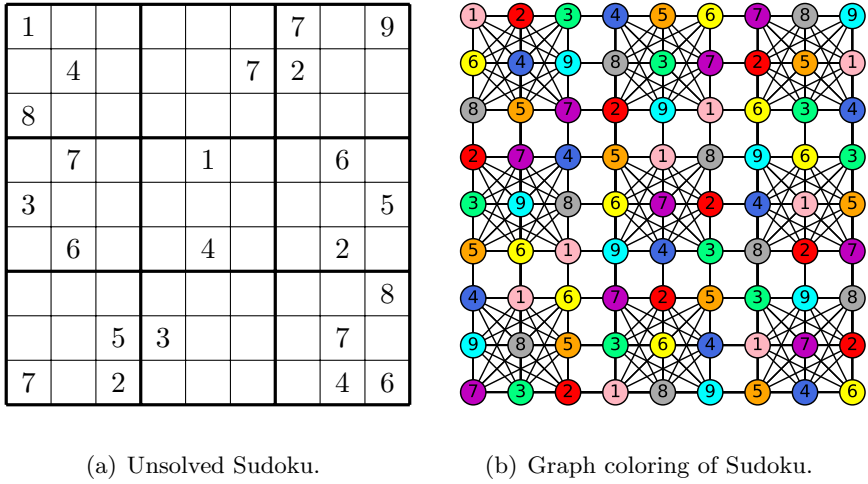
\begin{figure}[ht!]
	\centering
	
	\subfigure[Unsolved Sudoku.]{
		\begin{tikzpicture}[scale=0.59]
		\draw (0, 0) grid (9, 9);
		\draw[very thick, scale=3] (0, 0) grid (3, 3);
		
		\setcounter{row}{1}
		\setrow {1}{ }{ }  { }{ }{ }  {7}{ }{9}
		\setrow { }{4}{ }  { }{ }{7}  {2}{ }{ }
		\setrow {8}{ }{ }  { }{ }{ }  { }{ }{ }
		
		\setrow { }{7}{ }  { }{1}{ }  { }{6}{ }
		\setrow {3}{ }{ }  { }{ }{ }  { }{ }{5}
		\setrow { }{6}{ }  { }{4}{ }  { }{2}{ }
		
		\setrow { }{ }{ }  { }{ }{ }  { }{ }{8}
		\setrow { }{ }{5}  {3}{ }{ }  { }{7}{ }
		\setrow {7}{ }{2}  { }{ }{ }  { }{4}{6}
		
		\node[anchor=center] at (4.5, -0.5) {};
		
		\end{tikzpicture}}
	\subfigure[Graph coloring of Sudoku.]{\includegraphics[width=0.4\linewidth]{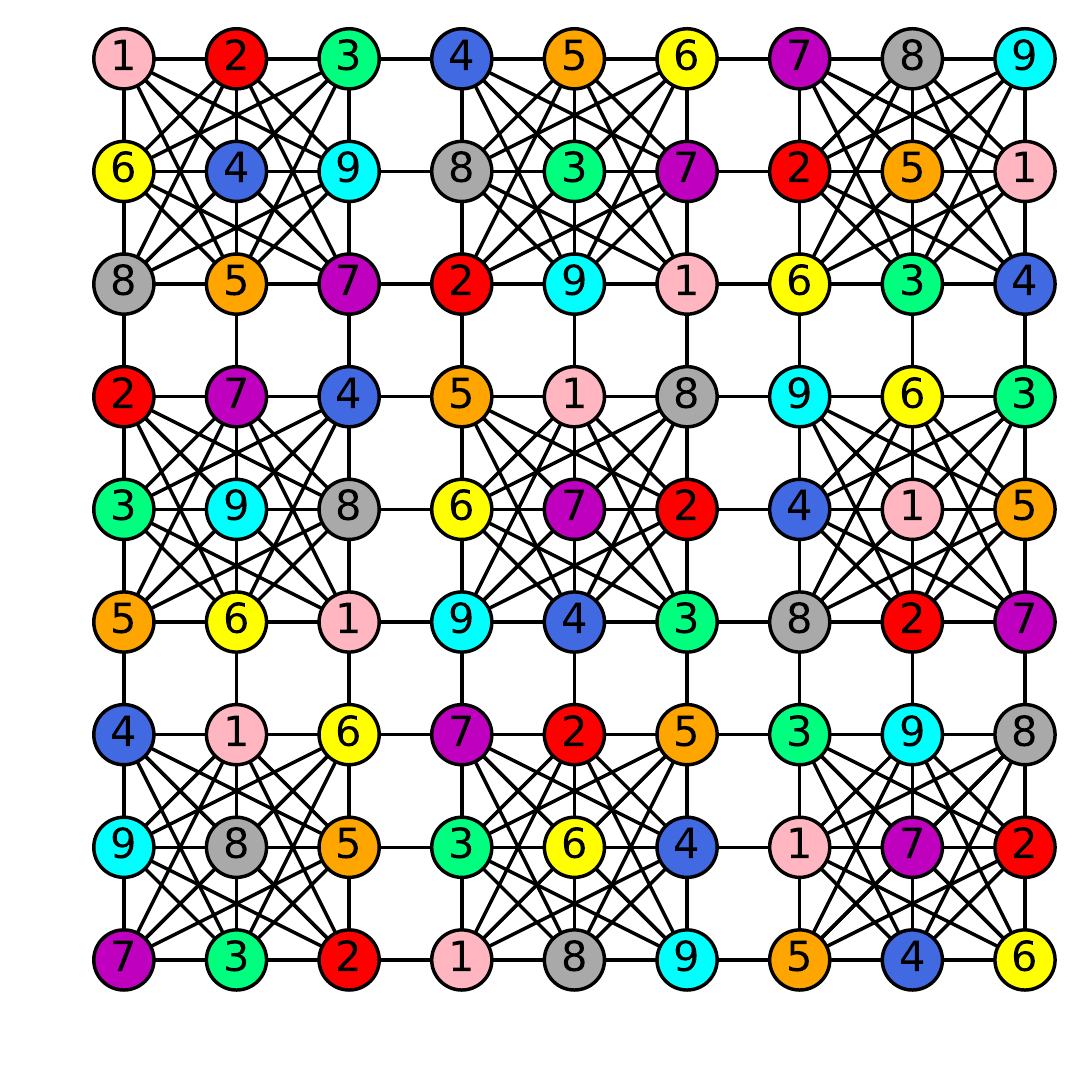}}	
	\caption{Sudoku solved by graph coloring.}\label{fig:Sudoku_colored}
\end{figure}

\section{The 8-queens puzzle and generalizations}\label{sec:8Q}

The \emph{8-queens puzzle} consists in placing eight chess queens on an $8\times8$ chessboard, so that none of them attack any other. Since a chess queen can be moved any number of squares vertically, horizontally or diagonally, the puzzle's constraints can be formulated as: there is at most one queen at each row, each column and each diagonal. The reformulation of an 8-queens puzzle as a graph coloring problem is similar to the one shown for Sudokus. Each square in the chessboard is represented by a node, and two nodes are linked if their corresponding squares lay on the same column, row or diagonal. The graph has $64$ nodes, $728$ links and $42$ maximal cliques.

To solve the 8-queens puzzle, it is not necessary to color all the nodes, but only $8$ of them with only one color. Thus, we are dealing with a partial graph coloring problem, in which we add the constraint that the color has to be used exactly $8$ times. We must then remove the set $C_4$ in~\eqref{eq:formulation_1} and replace the sets $C_1$ and $C_3$ by
\begin{gather*}
\widecheck C_1:=\left\{Z\in\R^{(n+l)\times m}:z_{ik}\in\{0,1\}, \forall (i,k)\in I\times K\text{ and } \sum_{k=1}^m z_{ik}\leq 1, \forall i\in I\right\},\\
\widecheck C_3:=\left\{Z\in\{0,1\}^{(n+l)\times m}:\sum_{i=1}^n z_{ik}=q, \forall k\in K\right\},
\end{gather*}
where $n=q=8$ and $m=1$ (puzzles with more colors can be considered). Hence, the solution set of the puzzle is $\widecheck C_1\cap C_2\cap\widecheck C_3$. The projections onto $\widecheck C_1$ and $\widecheck C_3$ are given by
\begin{gather*}
\left(P_{\widecheck C_1}(Z)\right)[i,k]=\left\{\begin{array}{ll} \min\left\{1,\max\{0,\round(z_{ik})\}\right\}&\text{if } i\in I,k=\arg\!\max\{z_{i1},z_{i2},\ldots,z_{im}\},\\
z_{ik}&\text{if } i\in P,\\
0&\text{otherwise};\end{array}\right.\\
\left(P_{\widecheck C_3}(Z)\right)[i,k]=\left\{\begin{array}{ll} 1&\text{if } i\in Q_{k,q},\\
\min\left\{1,\max\{0,\round(z_{ik})\}\right\}&\text{if } i\in P,\\
0&\text{otherwise;}\end{array}\right.
\end{gather*}
where, for a given color $k\in K$, we denote by $Q_{k,q}\subset I$ the set of indices corresponding to the $q$ largest values in $\{z_{1k}, z_{2k},\ldots, z_{nk}\}$ (lowest index is chosen in case of tie).

The 8-queens puzzle can be easily posed for any size of the chessboard. The problem has been generalized in many different directions, see~\cite{BS09} for a recent survey. One of these generalizations is the $n\text{-queens}^2$ puzzle, where one must cover an entire chessboard $n\times n$ with $n^2$ queens, so that two queens of the same color do not attach each other. This problem is actually the $n$-coloring problem of the chessboard queens graph, so it can be directly modeled as explained in Section~\ref{sec:Modeling_graphcoloring} using formulation~\eqref{eq:formulation_1}.
Different shapes can also be considered: we show in Figure~\ref{fig:puzzles}(b) a chessboard with a hole, and in Figure~\ref{fig:puzzles}(c) a puzzle dedicated to Jonathan Borwein. A solution to these puzzles, obtained with DR, is shown in Figure~\ref{fig:puzzles_solved}.

The use of the Douglas--Rachford algorithm for solving the $n$-queens puzzle is proposed and studied in~\cite{Schaad}. One of the main advantages of formulating these puzzles as graph coloring problems is that it is straightforward to model many variations of the problem. For instance, to model the knights puzzle, a similar puzzle played with knights instead of queens, one only needs to change the links of the chessboard graph, see Figure~\ref{fig:puzzles}(a).

\begin{figure}[ht!]
	\centering	
	\subfigure[Classic chessboard.]{\includegraphics[width=0.3\linewidth]{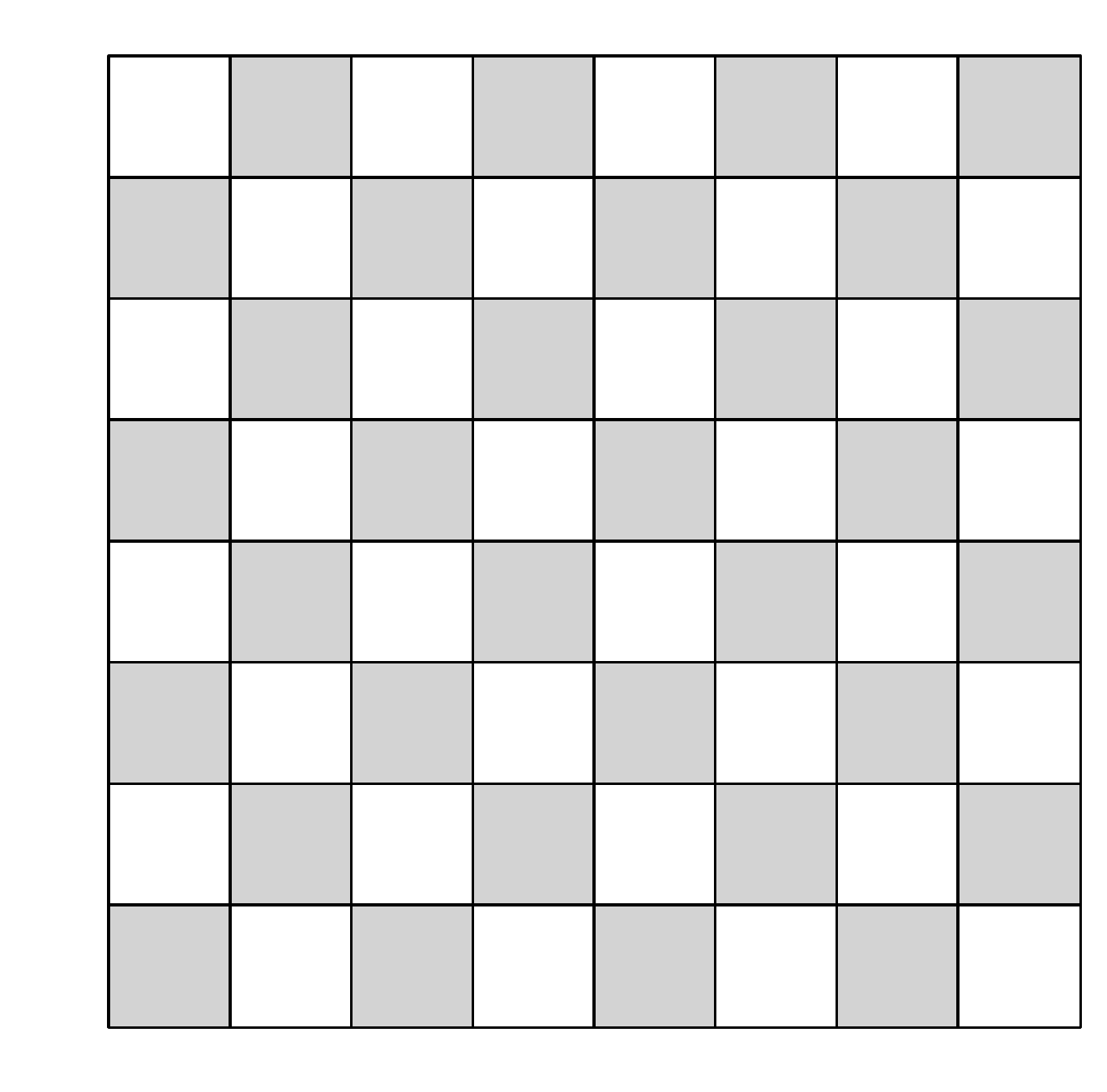}}\qquad
	\subfigure[Chessboard with a hole.]{\quad\includegraphics[width=0.3\linewidth]{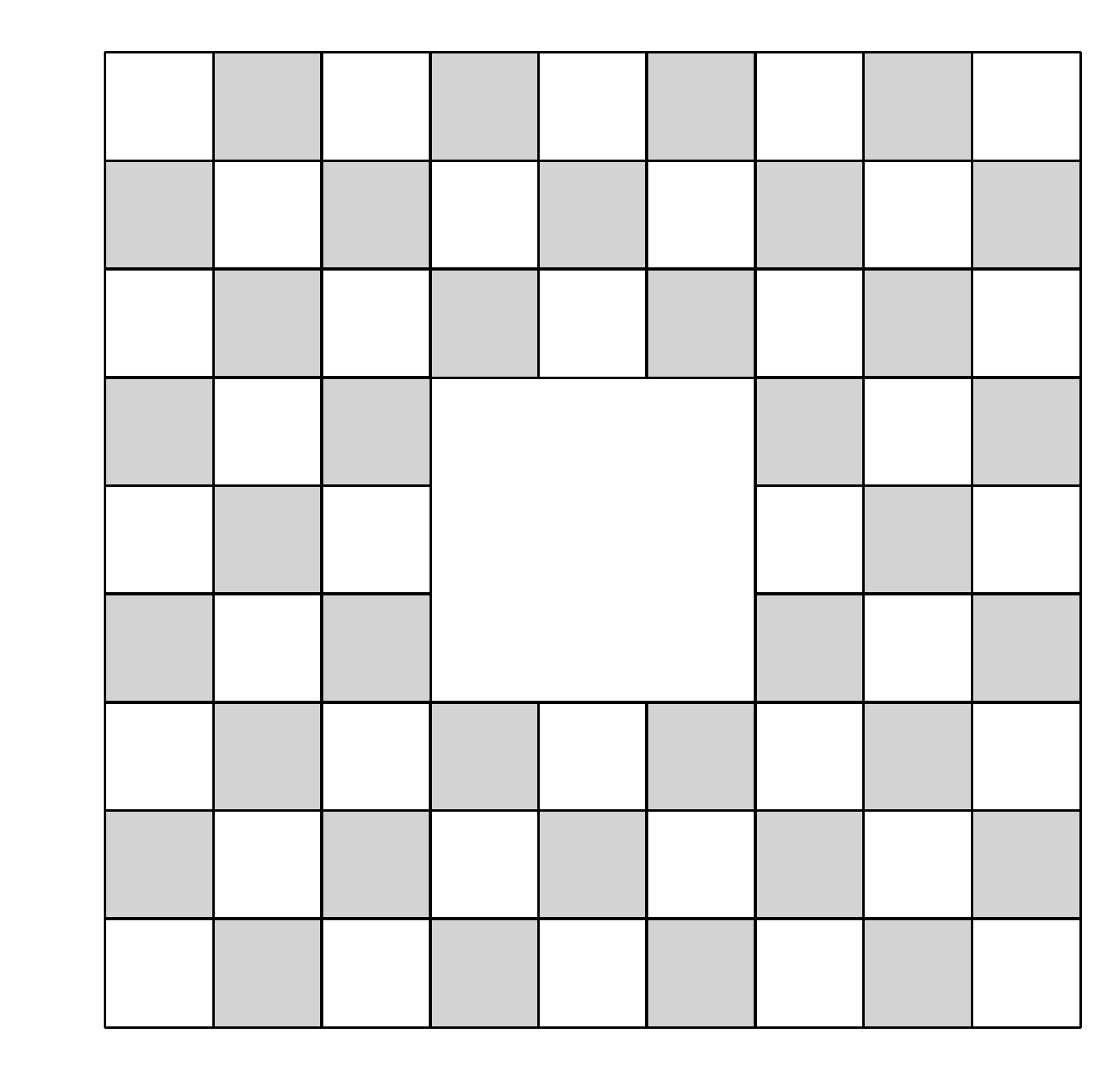}}
	\subfigure[`$\pi$-zzle'.]{\includegraphics[width=0.5\linewidth]{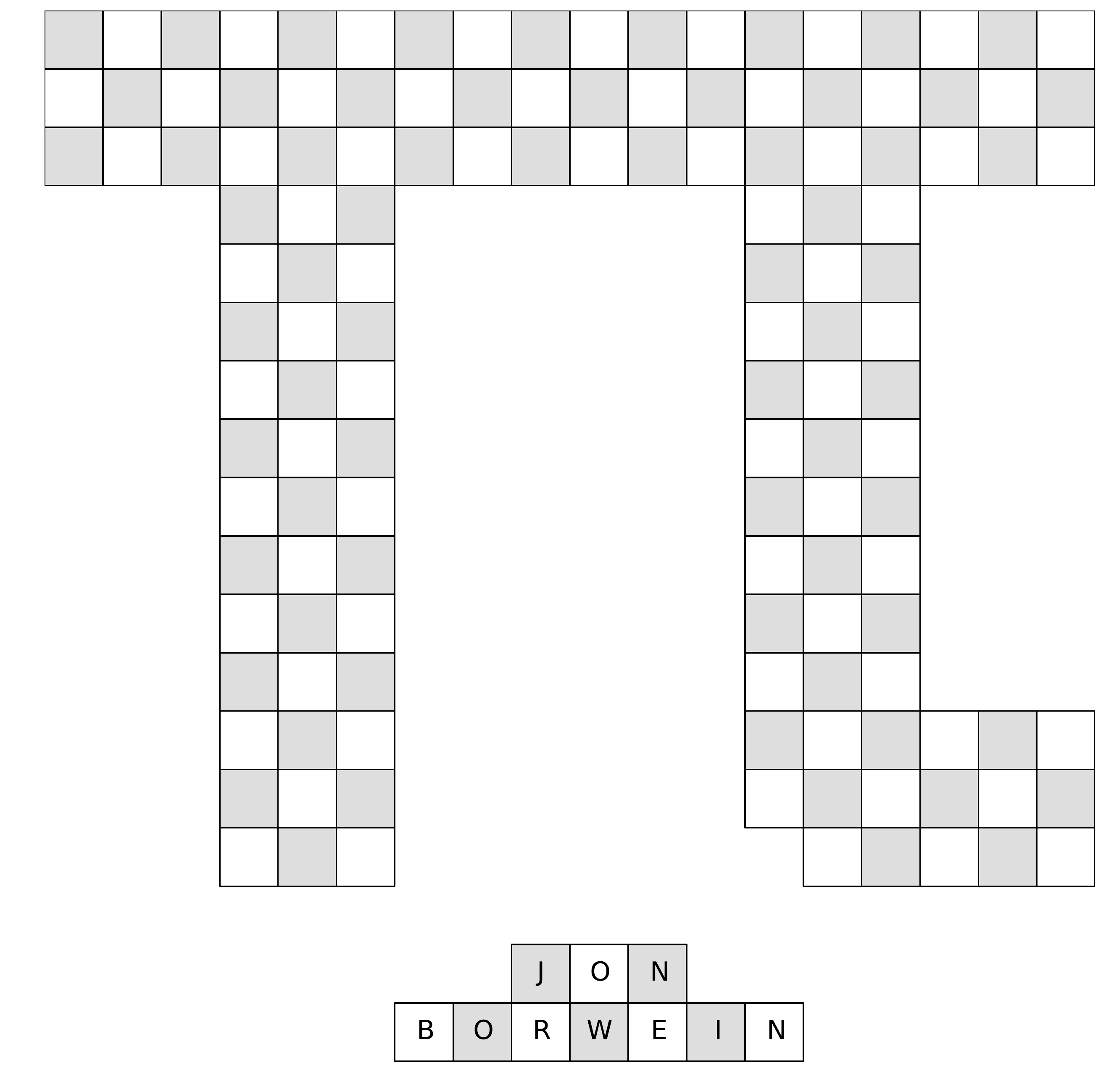}}	
	\caption{(a) A $16$-knights puzzle with $4$ colors: a solution will fill the chessboard. (b)~A 10-queens puzzle with $3$ colors played in a $9\times 9$ chessboard with a hole. (b) Empty `$\pi$-zzle'. The goal of this puzzle is to place on the board 8 times each of the 18 letters A, B, C, D, E, F, G, H, I, J, K, L, M, N, O, P, R and W. Ten cells have been prefilled. A solution to these puzzles computed with the Douglas--Rachford algorithm is shown in Figure~\ref{fig:puzzles_solved}.}\label{fig:puzzles}
\end{figure}

\section{The Hamiltonian path problem}\label{sec:hamiltonian}

A \emph{Hamiltonian path} is a path in a graph that visits every vertex exactly once. The Hamiltonian path problem consists in determining whether or not such a path exists. In this section we adapt the graph coloring scheme with the aim of using the Douglas--Rachford algorithm for finding Hamiltonian paths.

Given a graph $G$ with $n$ nodes, our objective will be to find an $n$-coloring of the graph, where each color $1, 2, \ldots, n$ will represent a position in the path. In order to ensure that the coloring represents a valid path, we will impose that two nodes assigned with two consecutive colors must be linked. Constraint $C_2$ becomes now redundant, as every node must be assigned with a different color, and it is thus no longer necessary to work in $\mathbb{R}^{(n+l)\times n}$, but in $\mathbb{R}^{n\times n}$. Hence, constraint $C_1$ becomes
$$\widetilde C_1:=\left\{X\in\R^{n\times n}:x_{ik}\in\{0,1\}, \forall (i,k)\in I\times K\text{ and } \sum_{k=1}^m x_{ik}=1, \forall i\in I\right\},$$
and the set $C_3$ must be modified and replaced by
\begin{gather*}
\widetilde C_3:=\left\{X\in\{0,1\}^{n\times n}: \forall k=1,\ldots,n-1, \exists \{i,j\}\in E \text{ s.t. } x_{i,k}x_{j,k+1}=1\right\}.
\end{gather*}
We have observed that the performance of DR is decreased if $C_2$ is removed, and that it is better to replace it by the redundant constraint $\widetilde C_2:=\mathbb{R}^{n\times n}$, see the experiment shown in Figure~\ref{fig:Exp_KnTour}. Note that constraint $C_4$ forces the path to start on node~$1$ (a path which may not even exist), so it must be eliminated.

The projection onto $\widetilde C_3$ is hard to compute because of the recurrent dependence between all the columns in the matrix $X$. To overcome this problem, we propose to split the set~$\widetilde C_3$ into two constraints, one relating each odd column with its following one, and another similar constraint for the even columns. That is, we define the constraints
\begin{gather*}
\widetilde C_{3,\text{odd}}:=\left\{X\in\{0,1\}^{(n+l)\times n}: \forall k=1,\ldots, \left\lfloor \frac{n}{2} \right\rfloor, \exists \{i,j\}\in E \text{ s.t. } x_{i,2k-1}x_{j,2k}=1\right\},\\
\widetilde C_{3,\text{even}}:=\left\{X\in\{0,1\}^{(n+l)\times n}: \forall k=1,\ldots,\left\lfloor \frac{n-1}{2} \right\rfloor, \exists \{i,j\}\in E \text{ s.t. } x_{i,2k}x_{j,2k+1}=1\right\},
\end{gather*}
which satisfy $\widetilde C_3=\widetilde C_{3,\text{odd}}\cap\widetilde C_{3,\text{even}}$, where $\lfloor \cdot \rfloor$ denotes the integer part of a number. Therefore, the solution set of the Hamiltonian path problem is $\widetilde C_1\cap \widetilde C_2\cap\widetilde C_{3,\text{odd}}\cap\widetilde C_{3,\text{even}}$.

To compute the projections onto $\widetilde C_{3,\text{odd}}$ and $\widetilde C_{3,\text{even}}$, consider the function $h:\R\mapsto \R$ defined by
\begin{equation*}
h(x):=\left\{\begin{array}{ll} x&\text{if } x\leq 0.5,\\
1&\text{if } x>0.5,\end{array}\right.
\end{equation*}
and let us denote by
$$(s_{k_1,k_2}^0,s_{k_1,k_2}^1)=\arg\!\min\left\{\left(1-h(x_{i,k_1})\right)^2+\left(1-h(x_{j,k_2})\right)^2, \{i,j\}\in E\right\},$$ where the lowest index is taken in $\arg\!\min$ to avoid multivaluedness. Then, the projections onto $\widetilde C_{3,\text{odd}}$ and $\widetilde C_{3,\text{even}}$ can be obtained as follows
\begin{gather*}
\left(P_{\widetilde C_{3,\text{odd}}}(Z)\right)[i,k]=\left\{\begin{array}{ll} 1&\text{if } i=s_{k,k+1}^0, k<n \text{ and } k \text{ is odd},\\
1&\text{if } i=s_{k-1,k}^1 \text{ and } k \text{ is even},\\
\min\left\{1,\max\{0,\round(x_{ik})\}\right\}&\text{otherwise;}\end{array}\right.\\
\left(P_{\widetilde C_{3,\text{even}}}(Z)\right)[i,k]=\left\{\begin{array}{ll} 1&\text{if } i=s_{k,k+1}^0, k<n \text{ and } k \text{ is even},\\
1&\text{if } i= s_{k-1,k}^1, 1<k \text{ and } k \text{ is odd},\\
\min\left\{1,\max\{0,\round(x_{ik})\}\right\}&\text{otherwise.}\end{array}\right.
\end{gather*}

\subsection{Hamiltonian cycles}\label{subsec:ham_cycles}

A \emph{Hamiltonian cycle} is a Hamiltonian path that is also a cycle, that is, there is a link connecting the last node in the path and the first one. The problem of finding such a cycle can be cast as a Hamiltonian path problem as we show next.

Given a graph $G=(V,E)$, select any node $v\in V$ and make a copy of it, i.e., create a new node $v'$ that is connected with all nodes linked to $v$. Then, create another two new nodes $t$ and $s$, and link $t$ with $v$ and $s$ with $v'$ (see Figure~\ref{fig:hamiltoniancycle}).
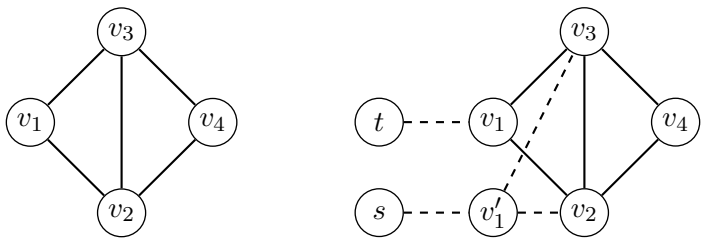
\begin{figure}[ht!]
	\centering
	\begin{tikzpicture}[scale=.6]%
	\GraphInit[vstyle=Normal]
	\SetVertexNoLabel
	\tikzstyle{EdgeStyle}=[thick]
	\Vertex[x=-2, y=2] {c0}
	\Vertex[x=0, y=0] {c1}
	\Vertex[x=0, y=4] {c2}
	\Vertex[x=2, y=2] {c3}
	\AssignVertexLabel{c}{$v_1$,$v_2$,$v_3$,$v_4$}
	\Edges(c0,c1,c3,c2,c0)
	\Edges(c1,c2)
	\end{tikzpicture}\qquad\qquad
	\begin{tikzpicture}[scale=.6]%
	\GraphInit[vstyle=Normal]
	\SetVertexNoLabel
	\tikzstyle{EdgeStyle}=[thick]
	\Vertex[x=-2, y=2] {c0}
	\Vertex[x=0, y=0] {c1}
	\Vertex[x=0, y=4] {c2}
	\Vertex[x=2, y=2] {c3}
	\Vertex[x=-2, y=0] {c4}
	\Vertex[x=-4.5, y=0] {s0}
	\Vertex[x=-4.5, y=2] {s1}
	\AssignVertexLabel{c}{$v_1$,$v_2$,$v_3$,$v_4$,$v'_1$}
	\AssignVertexLabel{s}{$s$,$t$}
	\Edges(c0,c1,c3,c2,c0)
	\Edges(c1,c2)
	\tikzstyle{EdgeStyle}=[dashed]
	\Edges(c4,c2)
	\Edges(c4,c1)
	\Edges(s0,c4)
	\Edges(s1,c0)
	\end{tikzpicture}
	\caption{Hamiltonian cycle reduced to Hamiltonian path.}\label{fig:hamiltoniancycle}
\end{figure}

Since $t$ and $s$ have \emph{degree one} (i.e., they are only linked with another node), every admissible Hamiltonian path in the new graph needs to start in one of these nodes and finish in the other. Thus, after removing $t$ and $s$, we end up with a path going from $v$ to $v'$. As these nodes were originally the same, we have actually found a Hamiltonian cycle.

\begin{example}
An example of Hamiltonian path/cycle arises in the \emph{knight's tour problem}. The \emph{knight's path problem} consists in finding a sequence of moves of a knight on a chessboard such that it visits exactly once every square. If the final position of such a path is one knight's move away from the starting position of the knight, the path is called a \emph{knight's cycle}. Thus, to find a knight's cycle, one only needs to build the graph corresponding to the knight's movements on a chessboard, and find a Hamiltonian cycle in the graph. A solution for a $12\times 12$ chessboard computed with DR is shown in Figure~\ref{fig:Knights_Tour}.
\begin{figure}[ht!]
	\centering
	\includegraphics[width=0.5\linewidth]{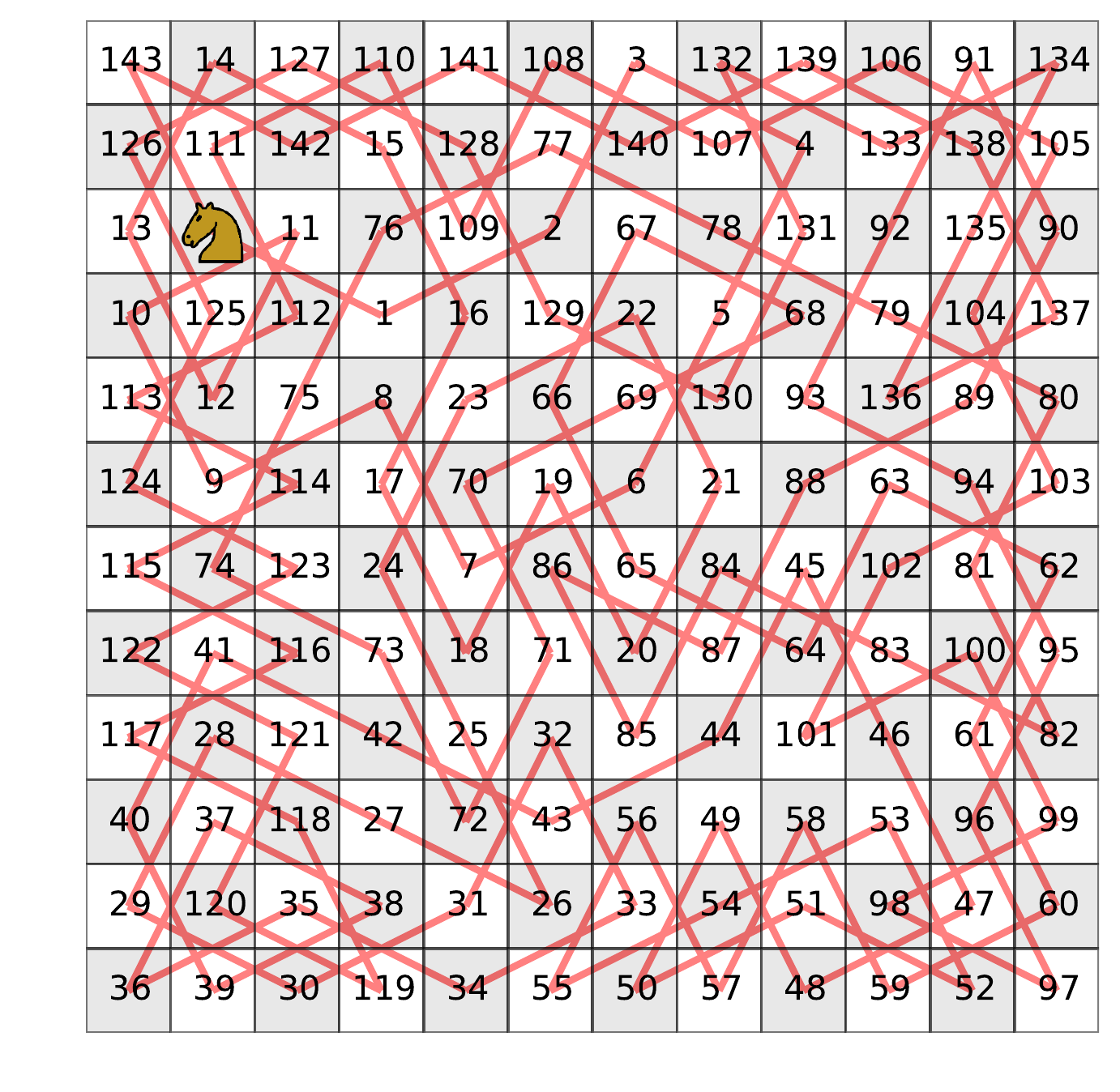}
	\caption{A knight's cycle on a $12\times12$ chessboard computed with DR. For 10~random starting
points, the method found a solution for every instance, with an average (maximum) time of 1,397 seconds (3,301 seconds, respectively).}\label{fig:Knights_Tour}
\end{figure}
\end{example}

\section{Numerical experiments}\label{sec:numerical}

In this section we test the performance of the Douglas--Rachford algorithm for solving a representative sample of the graph coloring problems previously presented. All codes are written in Python~2.7 and the tests were run on an Intel Core i7-4770 CPU \@3.40GHz with 12GB RAM, under Windows 10 (64-bits).

We begin our tests with one of the most well-known graphs: Petersen graph (see Figure~\ref{fig:Petersen}). This graph has 10 vertices, 15 edges and can be 3-colored in 120 different ways.

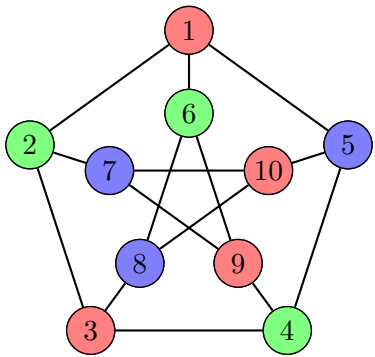
\begin{figure}[ht!]
\centering
\begin{tikzpicture}[scale=.55]%
    \GraphInit[vstyle=Normal]
    \SetVertexNoLabel
    \grPetersen[rotation=90,prefix=a,form=1,RA=4,RB=2]{};
    \AddVertexColor{red!50}{a0,a2,b3,b4}
    \AddVertexColor{green!50}{a1,a3,b0}
    \AddVertexColor{blue!50}{b1,b2,a4}
    \AssignVertexLabel{a}{1,...,5};
    \AssignVertexLabel{b}{6,...,10};
\end{tikzpicture}
\caption{A 3-coloring of Petersen graph.}\label{fig:Petersen}
\end{figure}

The results of our first experiment are shown in Figure~\ref{fig:Exp_Petersen}. For 100,000 random starting points and using formulation~\eqref{eq:formulation_1}, we report the number of iterations needed by the Douglas--Rachford algorithm until it obtained a solution. The success rate was 100\% in this experiment: for every starting point, the algorithm was able to find a solution.

\begin{figure}[ht!]
	\centering
\begin{minipage}{0.52\textwidth}
\centering
	\includegraphics[width=\linewidth]{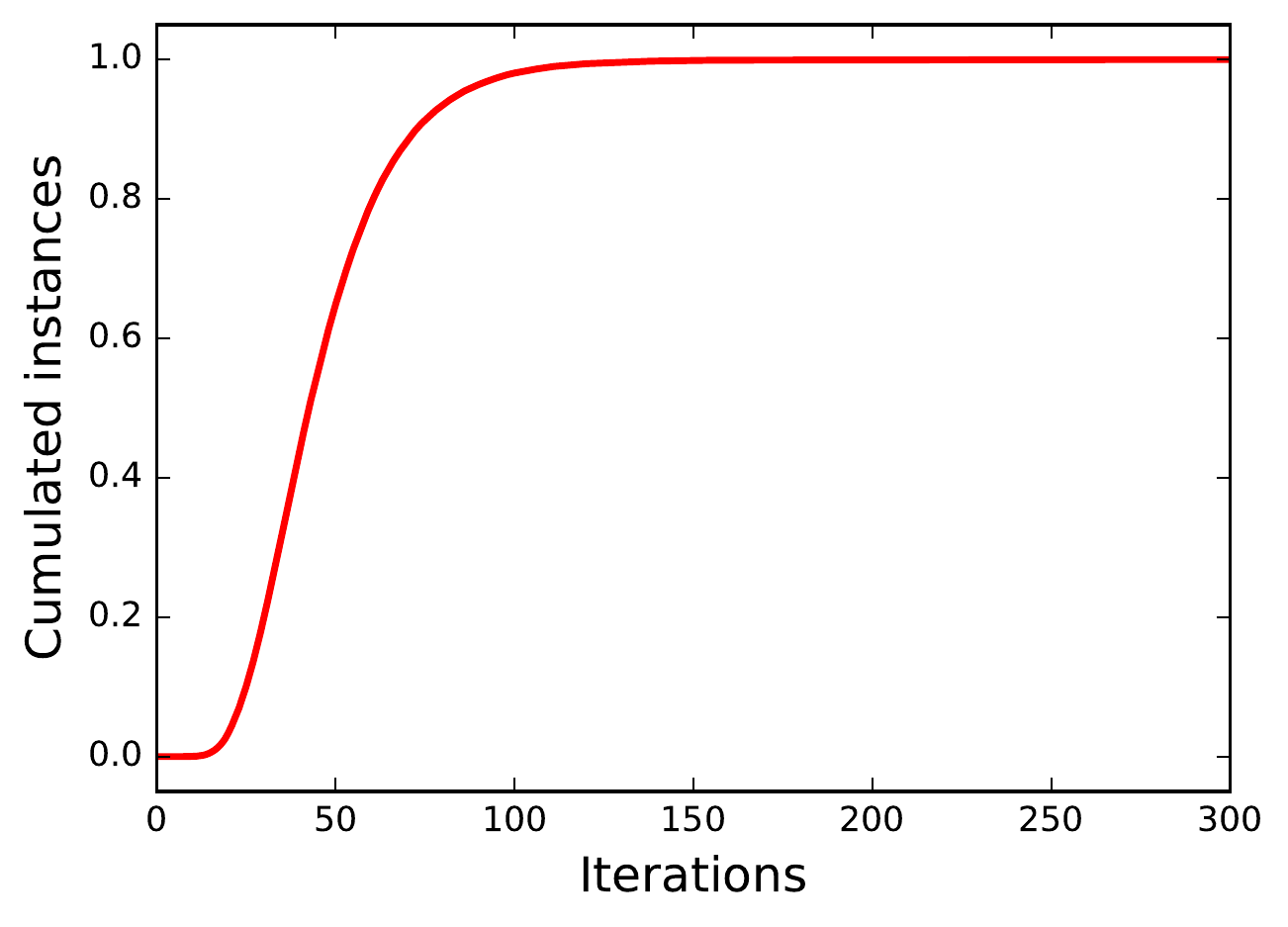}
\end{minipage}
\begin{minipage}{0.47\textwidth}
\centering

\begin{tabular}{|c||c|c|}
\hhline{-||--}
Iterations & Instances & Cumulated \tabularnewline
\hhline{=::==}
0-74 & 90,845 & 90.84\% \tabularnewline
75-149 & 9,037 & 99.88\% \tabularnewline
150-224 & 112 & 99.99\% \tabularnewline
225-299 & 6 & 100\% \tabularnewline
\hhline{-||--}
Unsolved & 0 & 100\% \tabularnewline
\hhline{-||--}
\end{tabular}
\end{minipage}

	\caption{Number of iterations spent by DR to find a solution of a 3-coloring of Petersen graph for 100,000 random starting points. On average, each solution was found in 0.00567 seconds.}
	\label{fig:Exp_Petersen}
\end{figure}

In our second experiment, we tested the performance of the Douglas--Rachford algorithm with formulation~\eqref{eq:formulation_1} for finding a valid coloring of complete graphs with 4, 5 and 6 nodes. A complete graph with $n$ vertices has $n(n-1)/2$ edges and can be $n$-colored in $n!$ different ways.  The algorithm was stopped after 500 iterations. DR was able to find a solution for every random starting point for the graphs of 5 and 6 nodes, while it failed in 0.16\% of the starting points for the complete graph of 4 nodes. The results are shown in Figure~\ref{fig:Exp_Complete456}.

\begin{figure}[ht!]
	\centering
	\includegraphics[width=0.65\linewidth]{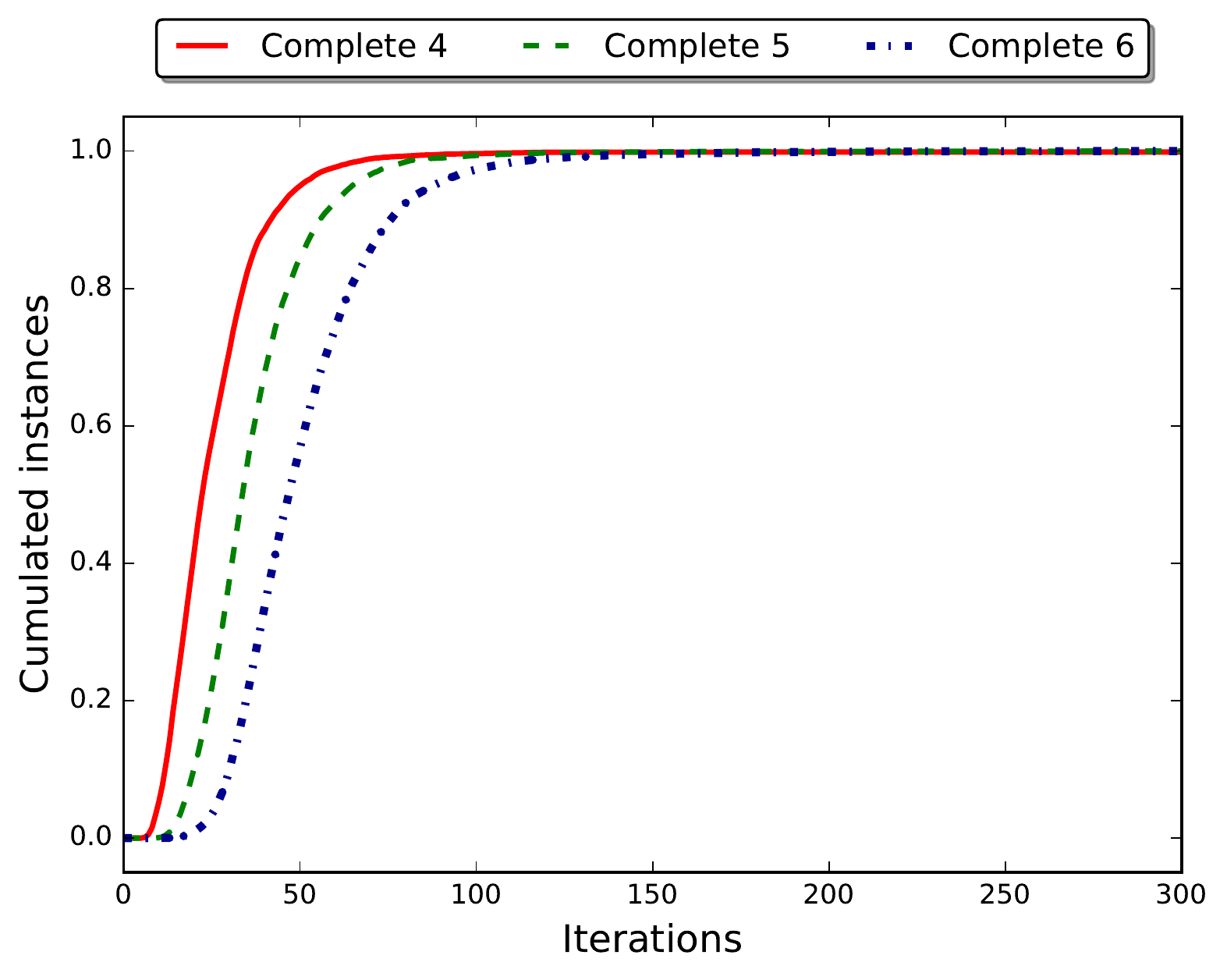}\vspace{5pt}
	
\centering

\begin{tabular}{|c||c|c||c|c||c|c|}
\cline{2-7}
\multicolumn{1}{c|}{} & \multicolumn{2}{c||}{Complete 4} & \multicolumn{2}{c||}{Complete 5} & \multicolumn{2}{c|}{Complete 6} \tabularnewline
\hhline{-||--||--||--}
Iterations & Instances & Cumul. & Instances & Cumul. & Instances & Cumul.\tabularnewline
\hhline{=::======}
0-99 & 9,961 & 99.61\% & 9,934 & 99.34\% & 9,718 & 97.18\% \tabularnewline
100-199 & 22 & 99.83\% & 61 & 99.95\% & 272 & 99.9\% \tabularnewline
200-299 & 1 & 99.84\% & 5 & 100\% & 10 & 100\% \tabularnewline
300-499 & 0 & 99.84\% & 0 & 100\% & 0 & 100\% \tabularnewline
\hhline{-||--||--||--}
Unsolved & 16 & 100\% & 0 & 100\% & 0 & 100\% \tabularnewline
\hhline{-||------}
\end{tabular}
\caption{Number of iterations spent by DR to find a solution of an $n$-coloring of a complete graph with $n$ vertices for 10,000 random starting points, with $n=4,5,6$. Each solution was found, on average, in 0.00281 seconds for $n=4$, 0.00429 seconds for $n=5$, and 0.00642 seconds for $n=6$. Instances were labeled as ``Unsolved'' after 500 iterations.}\label{fig:Exp_Complete456}
\end{figure}

We also tested the performance of the Douglas--Rachford algorithm on two wheel graphs of 5 and 6 nodes (see Figure~\ref{fig:wheel}). The results are shown in Figure~\ref{fig:Exp_Wheel56}. A wheel graph with $n$ vertices has $2(n-1)$ edges. If $n$ is even, it can be 4-colored in $4(2^{n-1}-2)$ different ways; if $n$ is odd, it can be 3-colored in 6 different ways.

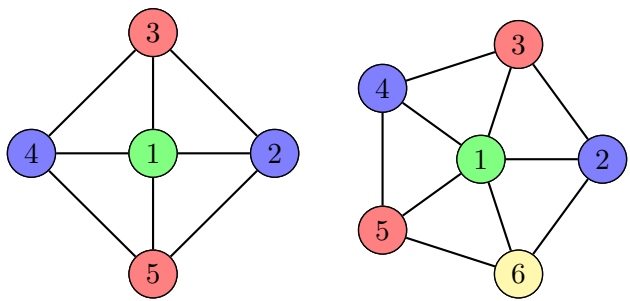
\begin{figure}[ht!]
\centering
\begin{tikzpicture}[scale=.4]%
    \GraphInit[vstyle=Normal]
    \SetVertexNoLabel
    \grWheel[prefix=a,form=1,RA=4,RB=2]{5};

    \AddVertexColor{red!50}{a1,a3}
    \AddVertexColor{green!50}{a4}
    \AddVertexColor{blue!50}{a2,a0}
    \AssignVertexLabel{a}{2,...,5,1};
\end{tikzpicture}\qquad
\begin{tikzpicture}[scale=.4]%
    \GraphInit[vstyle=Normal]
    \SetVertexNoLabel
    \grWheel[rotation=0,prefix=b,RA=4]{6};

    \AddVertexColor{red!50}{b1,b3}
    \AddVertexColor{green!50}{b5}
    \AddVertexColor{blue!50}{b2,b0}
    \AddVertexColor{yellow!40}{b4}
    \AssignVertexLabel{b}{2,...,6,1};
\end{tikzpicture}
\caption{A 3-coloring and a 4-coloring of two wheel graphs of 5 and 6 nodes.}\label{fig:wheel}
\end{figure}

\begin{figure}[ht!]
\begin{minipage}{0.46\textwidth}
	\centering
	\includegraphics[width=\linewidth]{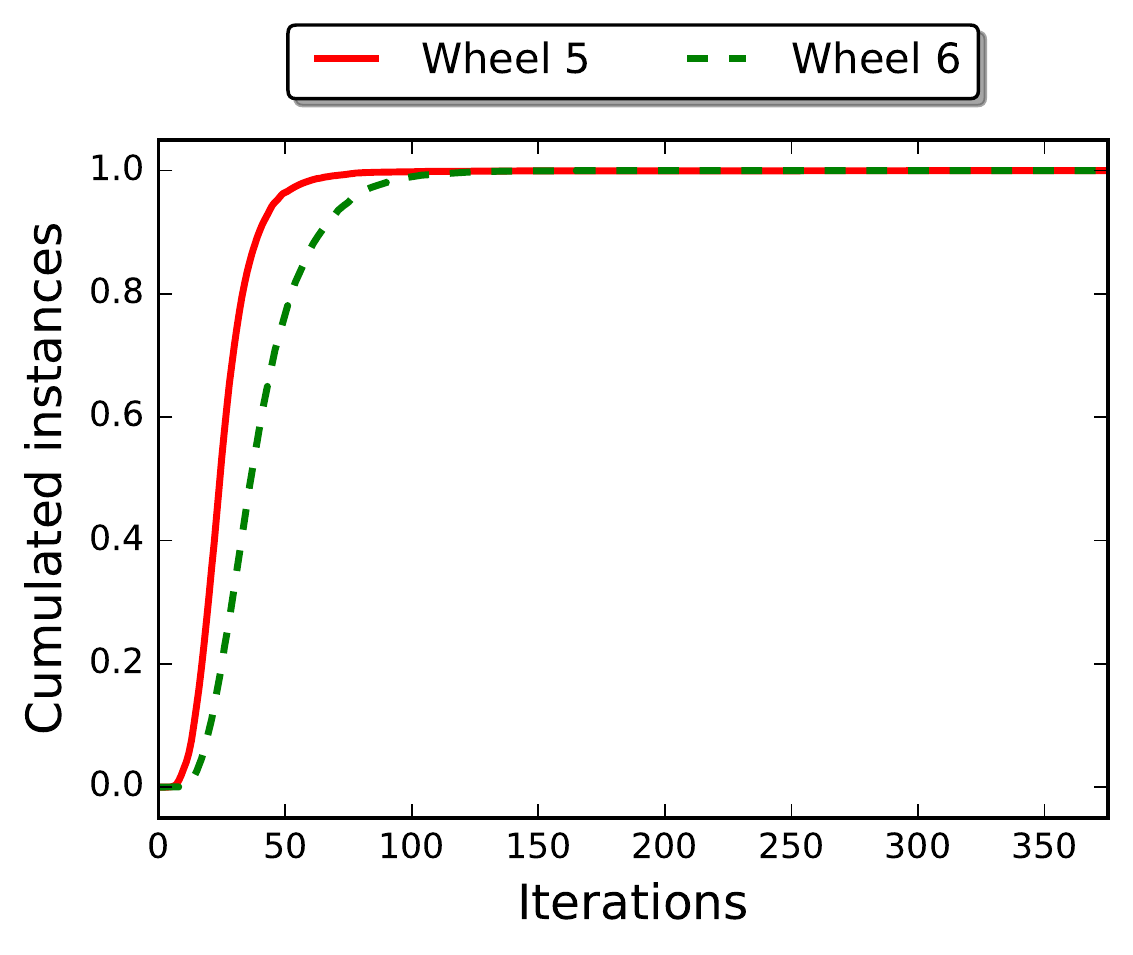}
\end{minipage}
\begin{minipage}{.53\textwidth}
\centering\small
\begin{tabular}{|c||c|c||c|c|}
\cline{2-5}
\multicolumn{1}{c|}{} & \multicolumn{2}{c||}{3-coloring} & \multicolumn{2}{c|}{4-coloring} \tabularnewline
\multicolumn{1}{c|}{} & \multicolumn{2}{c||}{of wheel 5} & \multicolumn{2}{c|}{of wheel 6} \tabularnewline
\hhline{-||--||--}
\!Iterations\! & Instan. & Cumul. & Instan. & Cumul. \tabularnewline
\hhline{=::====}
0-74 & 9,938 & 99.38\% & 9,455 & 94.55\% \tabularnewline
75-149 & 57 & 99.95\% & 541 & 99.96\% \tabularnewline
150-224 & 1 & 99.96\% & 4 & 100.0\% \tabularnewline
225-299 & 2 & 99.98\% & 0 & 100.0\% \tabularnewline
300-374 & 1 & 99.99\% & 0 & 100.0\% \tabularnewline
375-499 & 0 & 99.99\% & 0 & 100.0\% \tabularnewline
\hhline{-||--||--}
Unsolved & 1 & 100\% & 0 & 100\% \tabularnewline
\hhline{-||----}
\end{tabular}
\end{minipage}
\caption{Number of iterations spent by DR to find a solution of two wheel graphs for 10,000 random starting points. Each solution was found, on average, in 0.00291 seconds for wheel 5, and 0.00463 seconds for wheel 6. Instances were labeled as ``Unsolved'' after 500 iterations.}\label{fig:Exp_Wheel56}
\end{figure}

We repeated the same experiment with three cycle graphs (consisting in a given number of vertices connected in a closed chain). A cycle graph with $n$ vertices has $n$ edges. If $n$ is even, it can be 2-colored in 2 different ways; if $n$ is odd, it can be 3-colored in $2^n-2$ different ways. The results are shown in Figure~\ref{fig:Exp_Cycles101520}.

\begin{figure}[ht!]
	\centering
	\includegraphics[width=0.7\linewidth]{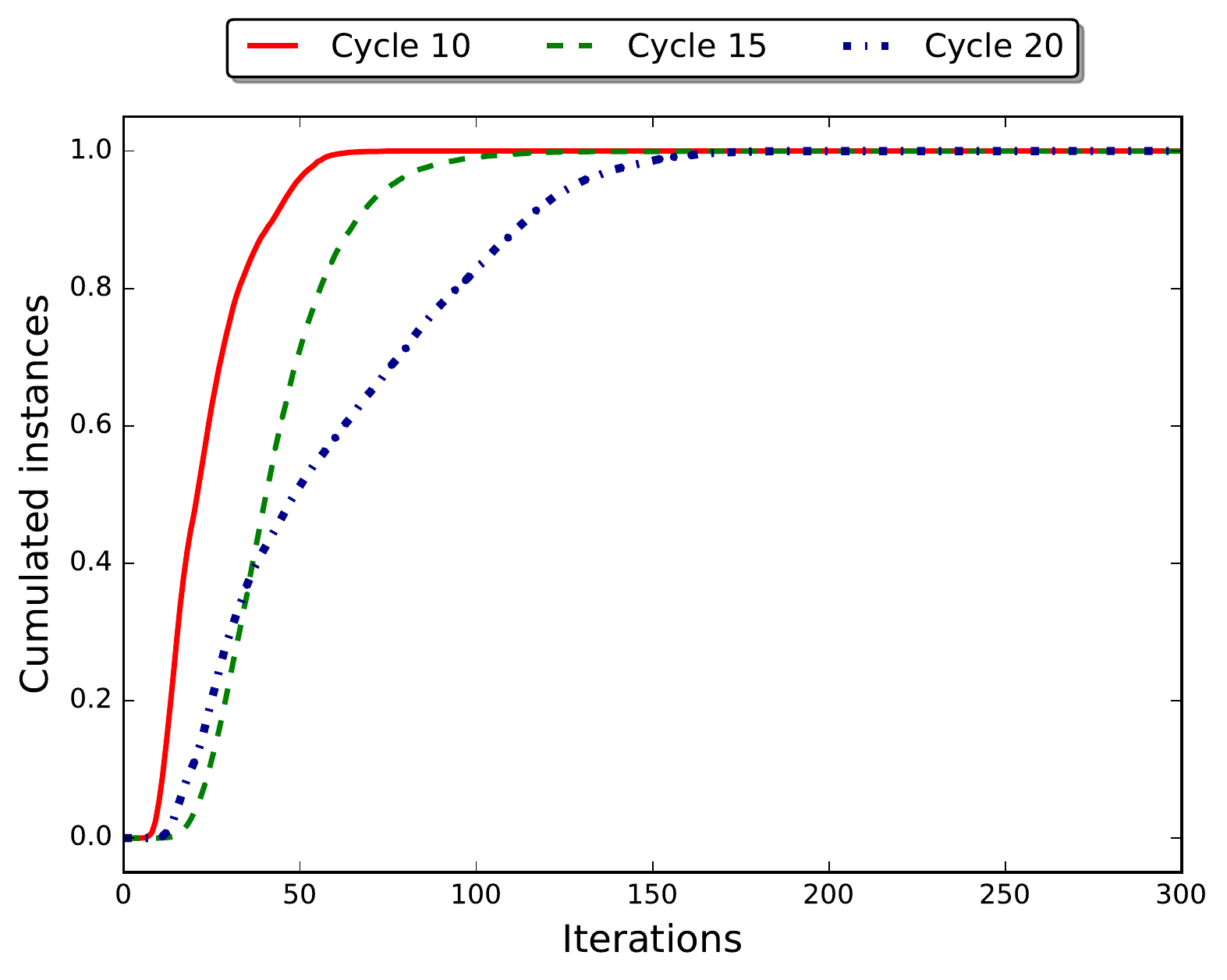}
\vspace{5pt}
	
\centering
\begin{tabular}{|c||c|c||c|c||c|c|}
\cline{2-7}
\multicolumn{1}{c|}{} & \multicolumn{2}{c||}{2-coloring} & \multicolumn{2}{c||}{3-coloring} & \multicolumn{2}{c|}{2-coloring} \tabularnewline
\multicolumn{1}{c|}{} & \multicolumn{2}{c||}{of cycle 10} & \multicolumn{2}{c||}{of cycle 15} & \multicolumn{2}{c|}{of cycle 20} \tabularnewline
\hhline{-||--||--||--}
Iterations & Instances & Cumul.  & Instances & Cumul.  & Instances & Cumul. \tabularnewline
\hhline{=::======}
0-99 & 10,000 & 100.0\% & 9,903 & 99.03\% & 8,251 & 82.51\% \tabularnewline
100-199 & 0 & 100.0\% & 94 & 99.97\% & 1,748 & 99.99\% \tabularnewline
200-299 & 0 & 100.0\% & 0 & 99.97\% & 1 & 100.0\% \tabularnewline
300-499 & 0 & 100.0\% & 0 & 99.97\% & 0 & 100.0\% \tabularnewline
\hhline{-||--||--||--}
Unsolved & 0 & 100\% & 3 & 100\% & 0 & 100\% \tabularnewline
\hhline{-||------}
\end{tabular}
\caption{Number of iterations spent by DR to find a solution of three cycle graphs for 10,000 random starting points. Each solution was found, on average, in 0.00277 seconds for cycle 10, 0.00568 seconds for cycle 15, and 0.00794 seconds for cycle 20. Instances were labeled as ``Unsolved'' after 500 iterations.}\label{fig:Exp_Cycles101520}
\end{figure}

In our following experiment, whose results are shown in Figure~\ref{fig:Exp_WindmillCliques}, we compare the performance of the Douglas--Rachford algorithm with and without maximal clique information when it is applied for finding a solution of the windmill graph $\Wd(10,5)$. Observe that, even having increased the number of variables in the feasibility problem, both the rate of success and the rate of convergence (in terms of iterations, but also computing time) are improved.

\begin{figure}[ht!]
	\centering
	\includegraphics[width=0.65\linewidth]{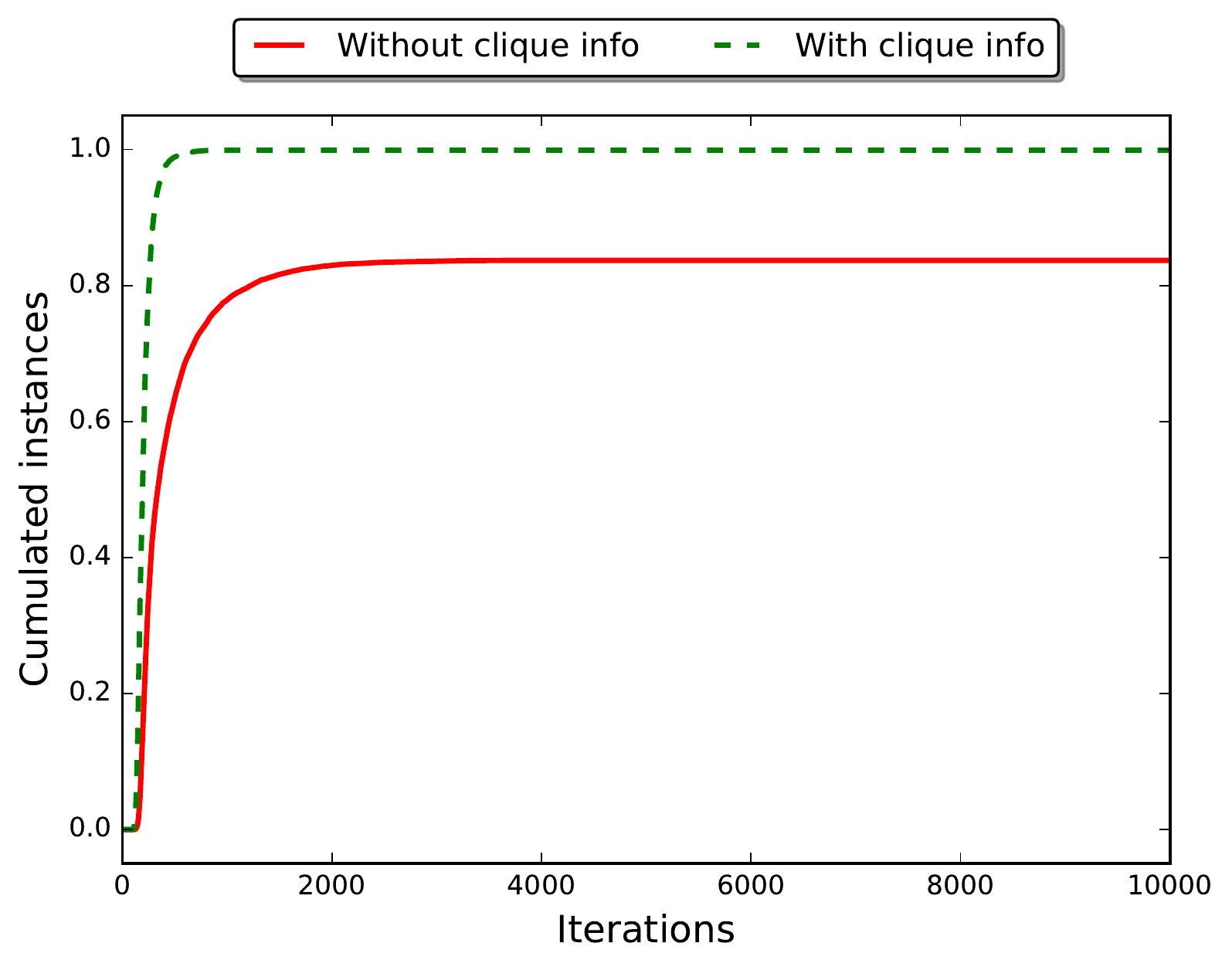}\vspace{5pt}

\centering
\begin{tabular}{|c||c|c||c|c|}
\cline{2-5}
\multicolumn{1}{c|}{} & \multicolumn{2}{c||}{Without maximal} & \multicolumn{2}{c|}{With maximal} \tabularnewline
\multicolumn{1}{c|}{} & \multicolumn{2}{c||}{clique information} & \multicolumn{2}{c|}{clique information} \tabularnewline
\hhline{-||--||--}
Iterations & Instances & Cumul. & Instances & Cumul. \tabularnewline
\hhline{=::====}
0-499 & 6,338 & 63.38\% & 9,887 & 98.87\% \tabularnewline
500-999 & 1,449 & 77.87\% & 101 & 99.88\% \tabularnewline
1,000-1,499 & 375 & 81.62\% & 1 & 99.89\% \tabularnewline
1,500-1,999 & 134 & 82.96\% & 0 & 99.89\% \tabularnewline
2,000-2,499 & 42 & 83.38\% & 0 & 99.89\% \tabularnewline
2,500-2,999 & 18 & 83.56\% & 0 & 99.89\% \tabularnewline
3,000-3,499 & 11 & 83.67\% & 0 & 99.89\% \tabularnewline
3,500-3,999 & 2 & 83.69\% & 0 & 99.89\% \tabularnewline
4,000-4,499 & 1 & 83.7\% & 0 & 99.89\% \tabularnewline
4,500-9,999 & 0 & 83.7\% & 0 & 99.89\% \tabularnewline
\hhline{-||--||--}
Unsolved & 1,630 & 100\% & 11 & 100\% \tabularnewline
\hhline{-||----}
\end{tabular}
\caption{Comparison of the number of iterations spent by DR to find a solution of the windmill graph $\Wd(10,5)$ for 10,000 random starting points. Complete maximal clique information was used in the right columns. Each solution was found, on average, in 0.1673 seconds without clique information, and 0.08484 seconds with maximal clique information. Instances were labeled as ``Unsolved'' after 10,000 iterations.}\label{fig:Exp_WindmillCliques}
\end{figure}

If $\|Z_k\|$ increases as $k$ increases, the Douglas--Rachford algorithm may serve to detect infeasibility of the corresponding coloring problem, see Figure~\ref{fig:infeasible}(a)-(b).  This is not always the case, as shown in Figure~\ref{fig:infeasible}(c)-(d). Interestingly, when we removed the extra constraints~\eqref{eq:C5} and~\eqref{eq:C4}, which is something that does not change the feasibility of the problems, the algorithm was not able to detect any infeasible problem.

\begin{figure}[ht!]
\centering
	 \subfigure[3-coloring of Petersen graph (feasible).]{\includegraphics[width=0.43\textwidth]{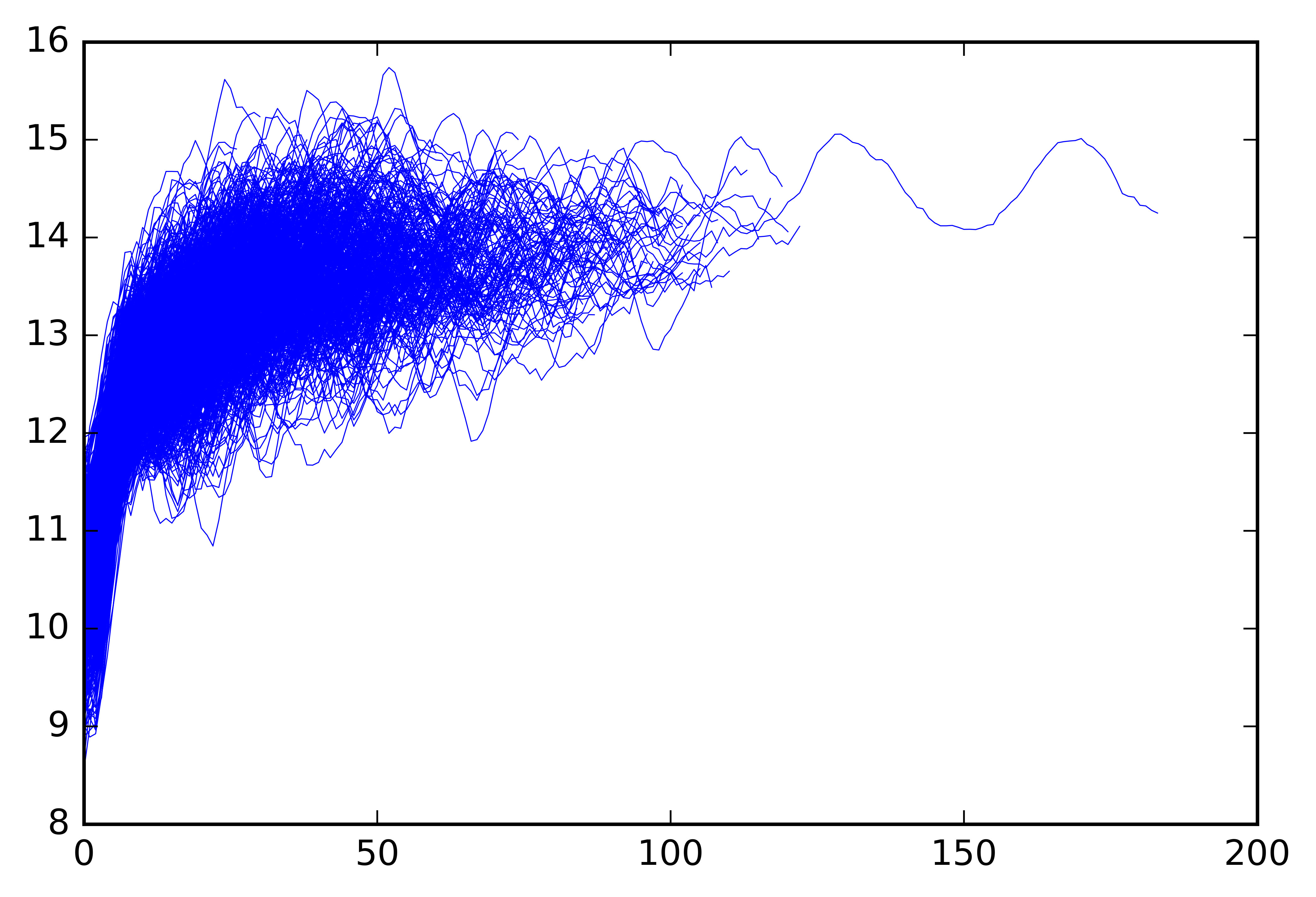}}\hspace{0.05\textwidth}
	 \subfigure[2-coloring of Petersen graph (infeasible).]{\includegraphics[width=0.43\textwidth]{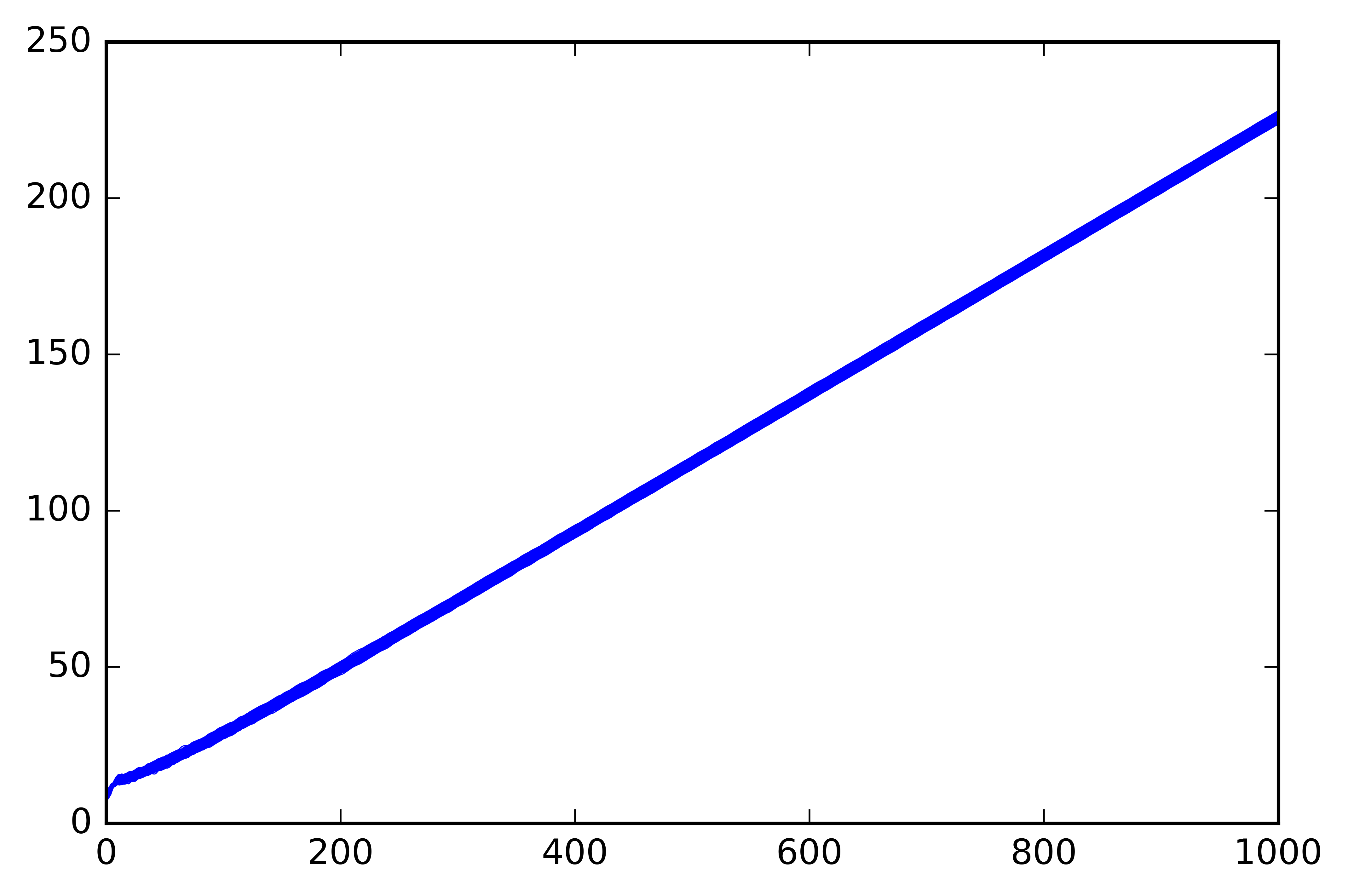}}

	 \subfigure[3-coloring of a 6-wheel graph (infeasible).]{\includegraphics[width=0.43\textwidth]{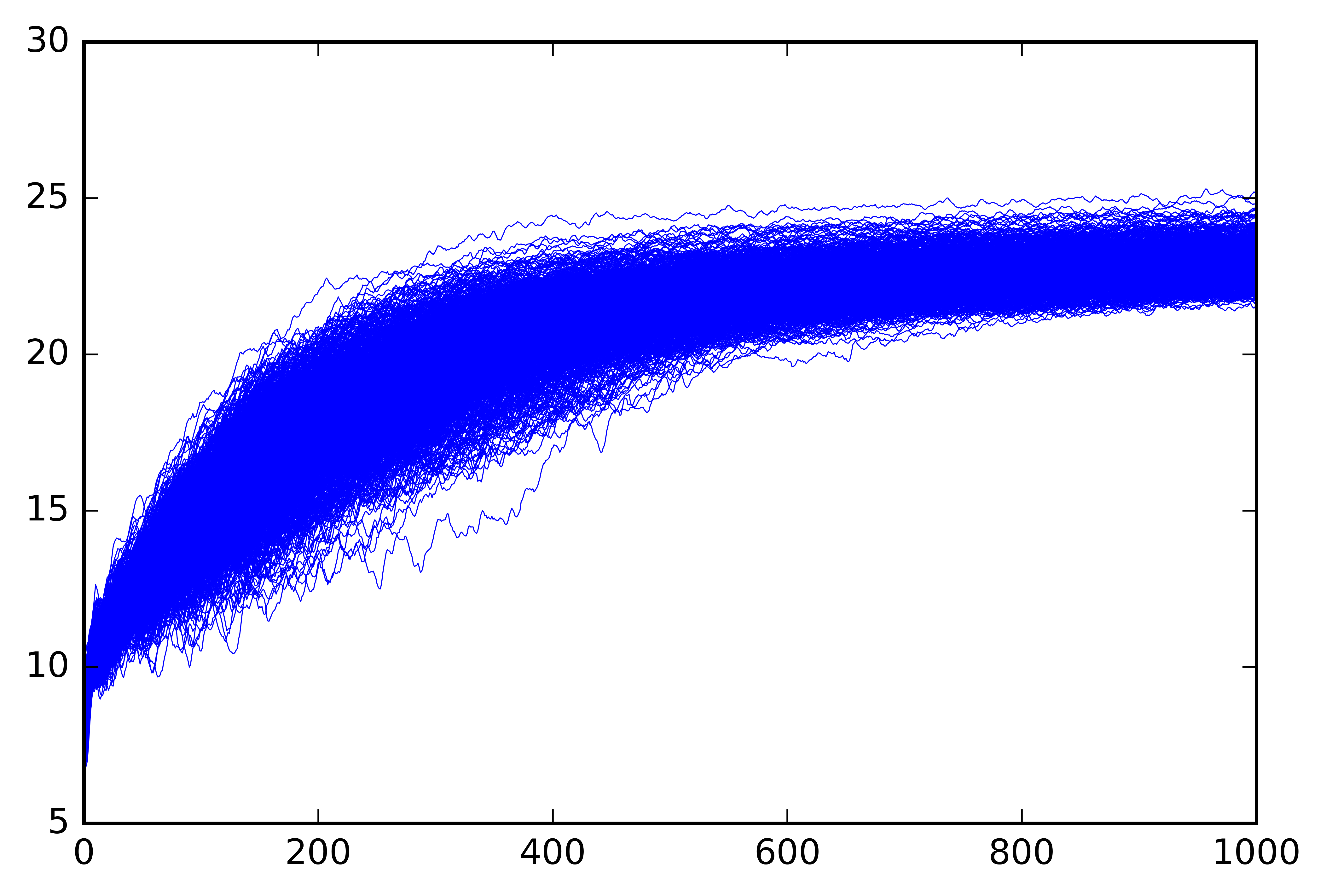}}\hspace{0.05\textwidth}
	 \subfigure[2-coloring of a 6-wheel graph (infeasible).]{\includegraphics[width=0.43\textwidth]{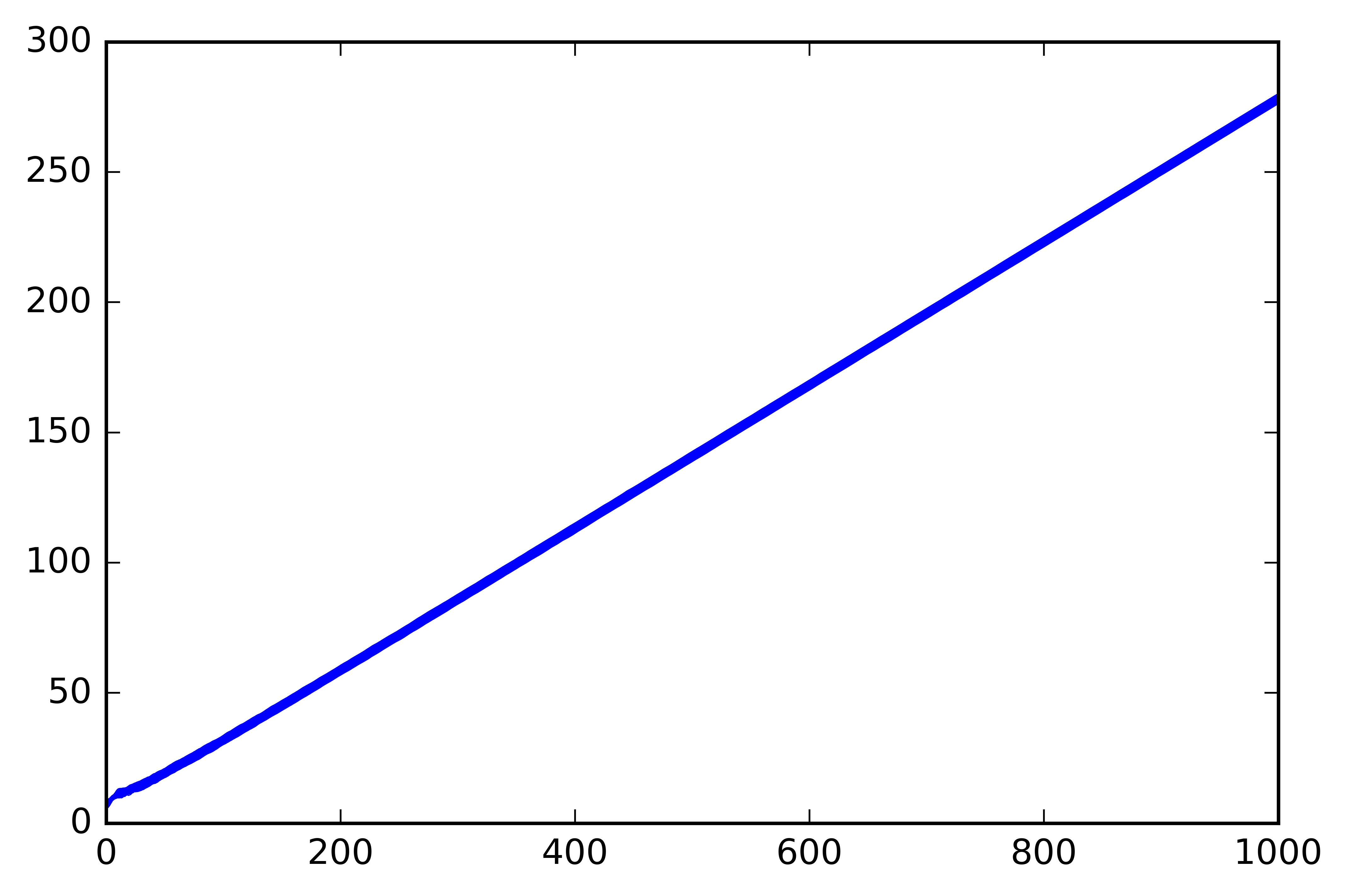}}
    \caption{For 1,000 random starting points, we represent the iteration $k$ in the horizontal axis and $\|Z_k\|$ in the vertical axis for 1,000 iterations of the Douglas--Rachford algorithm.}\label{fig:infeasible}
\end{figure}

Next, we tested the performance of DR for the 4-nodes and the 5-nodes formulations for the first 50 3-SAT problems with 20 variables and 91 clauses in \texttt{SATLIB}\footnote{\texttt{SATLIB}: \url{www.cs.ubc.ca/~hoos/SATLIB/Benchmarks/SAT/RND3SAT/uf20-91.tar.gz}}.
For each of the formulations, we run the Douglas--Rachford algorithm with and without maximal clique information for 10 random starting points. The results are shown in Table~\ref{tbl:3-SAT}. Clearly, the addition of the maximal clique information turns out to be crucial for the success of the Douglas--Rachford algorithm, specially for the 5-nodes formulation.

\begin{table}[ht!]
\centering
\begin{tabular}{|c||c|c||c|c||c|c||c|c|}
\cline{2-9}
\multicolumn{1}{c|}{} & \multicolumn{2}{c||}{4-nodes without} & \multicolumn{2}{c||}{4-nodes with}  & \multicolumn{2}{c||}{5-nodes without} & \multicolumn{2}{c|}{5-nodes with} \tabularnewline
\multicolumn{1}{c|}{} & \multicolumn{2}{c||}{clique info.} & \multicolumn{2}{c||}{clique info.} & \multicolumn{2}{c||}{clique info.} & \multicolumn{2}{c|}{clique info.}\tabularnewline
\hhline{-||--||--||--||--}
Time & Inst. & Cumul.& Inst. & Cumul. & Inst. & Cumul. & Inst. & Cumul. \tabularnewline
\hhline{=::========}
0-49 & 246 & 49\% & 341 & 68\% & 0 & 0\% & 295 & 59\% \tabularnewline
50-99 & 76 & 64\% & 69 & 82\% & 0 & 0\% & 77 & 74\% \tabularnewline
100-149 & 38 & 72\% & 20 & 86\% & 0 & 0\% & 22 & 78\% \tabularnewline
150-199 & 14 & 74\% & 19 & 89\% & 0 & 0\% & 9 & 80\% \tabularnewline
200-249 & 13 & 77\% & 7 & 91\% & 0 & 0\% & 5 & 81\% \tabularnewline
250-299 & 7 & 78\% & 7 & 92\% & 0 & 0\% & 4 & 82\% \tabularnewline
\hhline{-||--||--||--||--}
Unsolved & 106 & 100\% & 37 & 100\% & 500 & 100\% & 88 & 100\% \tabularnewline
\hhline{-||--------}
\end{tabular}
\caption{Time spent (in seconds) by DR to find a solution of 50 different 3-SAT problems with 20 variables and 91 clauses. For each problem, 10 random starting points were chosen. After 5 minutes without finding a solution, instances where labeled as ``Unsolved''. Two formulations of the gadgets were considered, with 4 and 5 nodes.}\label{tbl:3-SAT}
\end{table}

For an appropriate visualization of the results and comparison of the formulations, we turn to performance profiles (see~\cite{DM02}). We use the modification proposed in~\cite{ISU16}, since it suits better our experiment, where we have multiple runs for every formulation and problem. Let $\Phi$ denote the (finite) set of all formulations. For each formulation $f\in\Phi$, let $t_{f,p}$ be the average time required by DR to solve problem $p$ among all the successful runs, and let us denote by $s_{f,p}$ the portion of successful runs for problem $p$. Compute $t^\star_p:=\min_{f\in\Phi} t_{f,p}$ for all $p\in\{1,\ldots,n_p\}$, where $n_p$ is the number of problems in the experiment. Then, for any $\tau\geq 1$, define $R_f(\tau):=\{p\in\{1,\ldots,n_p\}, t_{f,p}\leq\tau t^\star_p\}$; that is, $R_f(\tau)$ is the set of problems for which formulation $f$ is at most $\tau$ times slower than the best one. The \emph{performance profile} function of formulation $f$ is given by
\begin{equation*}
\begin{array}{rccl}
\pi_f: & [1,+\infty) & \longmapsto & [0,1]\\
& \tau &  \mapsto & \pi_f(\tau):=\frac{1}{n_p}\sum_{p\in R_f(\tau)} s_{f,p}.
\end{array}
\end{equation*}

The value $\pi_f(1)$ indicates the portion of runs for which $f$ was the fastest formulation. When $\tau\rightarrow+\infty$, then $\pi_f(\tau)$ shows the portion of successful runs for formulation $f$. Performance profiles for the $3$-SAT experiment from Table~\ref{tbl:3-SAT} are displayed in Figure~\ref{fig:3SAT_PProfile}. It clearly shows that the 4-nodes formulation with clique information is the best one.

\begin{figure}[ht!]
	\centering
	\includegraphics[width=0.66\linewidth]{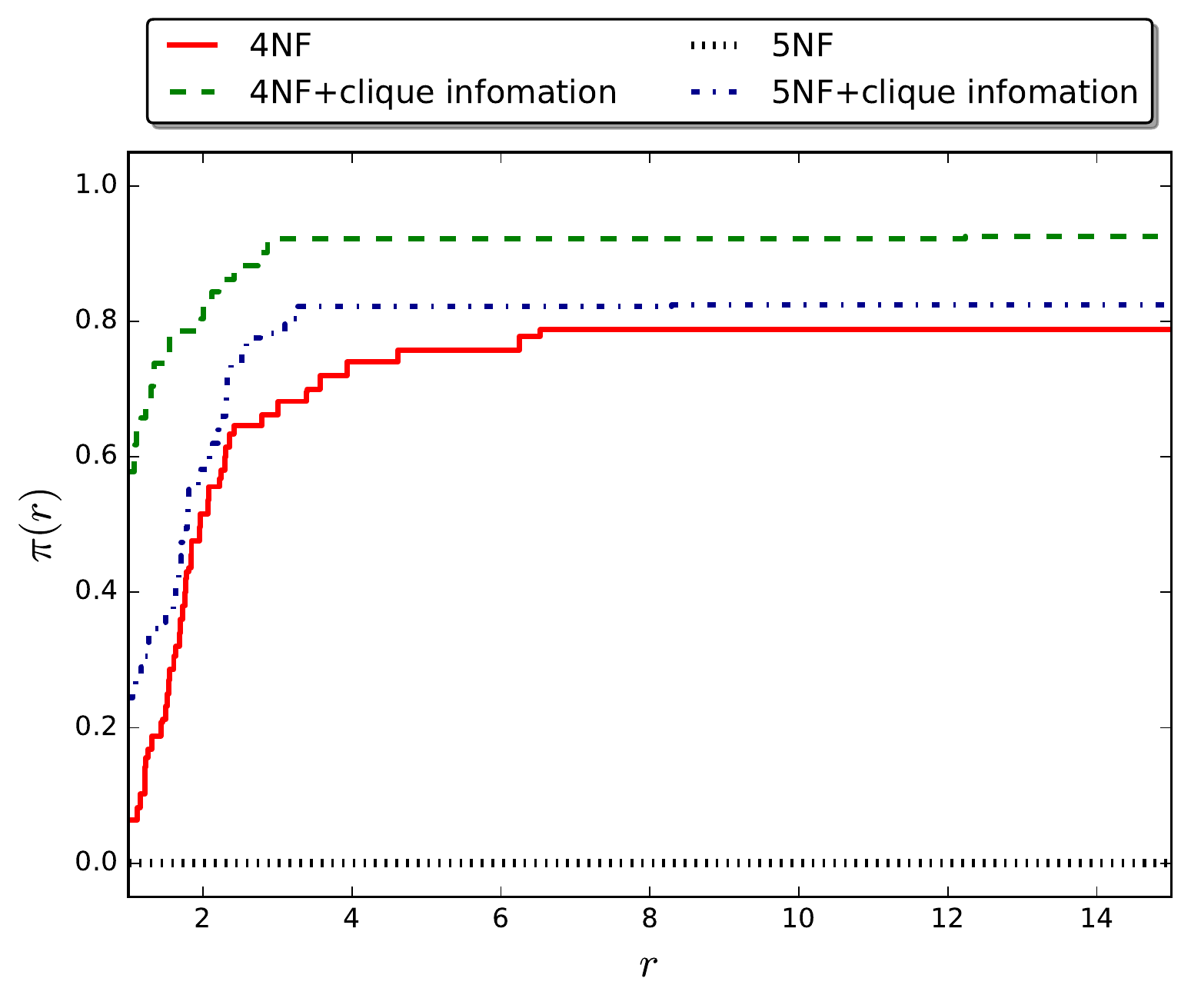}
	\caption{Performance profile functions for the results in the $3$-SAT experiment.}
	\label{fig:3SAT_PProfile}
\end{figure}

In our next numerical experiment, for solving Sudoku puzzles, we compared the performance  of DR applied to the Elser's binary feasibility problem formulation~\cite{Elser} (see also~\cite[Section~6.2]{ABTcomb}), with the reformulations as a graph coloring ($C_1\cap C_2\cap C_3\cap C_4^\star$) and as a graph precoloring ($\overline{C}_1\cap C_2\cap C_3$) explained in Section~\ref{sec:precoloring}. We considered the 95~hard puzzles from the library \texttt{top95}\footnote{\texttt{top95}: \url{http://magictour.free.fr/top95}}, which was the one among the libraries tested in~\cite[Table~2]{ABTcomb} where DR was most unsuccessful. For each formulation and each puzzle, Douglas--Rachford was run for 20 random starting points. Results and performance profiles are displayed in Figure~\ref{fig:Sudoku}. As it was expected, the binary formulation was much faster, since this formulation is specifically designed for solving these puzzles. On average, the binary formulation solved a Sudoku in 5.76 seconds, while the graph coloring formulation needed 33.78 seconds. The worst method was the reformulation as a graph coloring problem, which needed 112.25 seconds on average to solve a Sudoku. Even so, it was surprising to see that the rate of success for these three formulations was very similar, around 90\%.

\begin{figure}[ht!]
	\centering
	\includegraphics[width=0.66\linewidth]{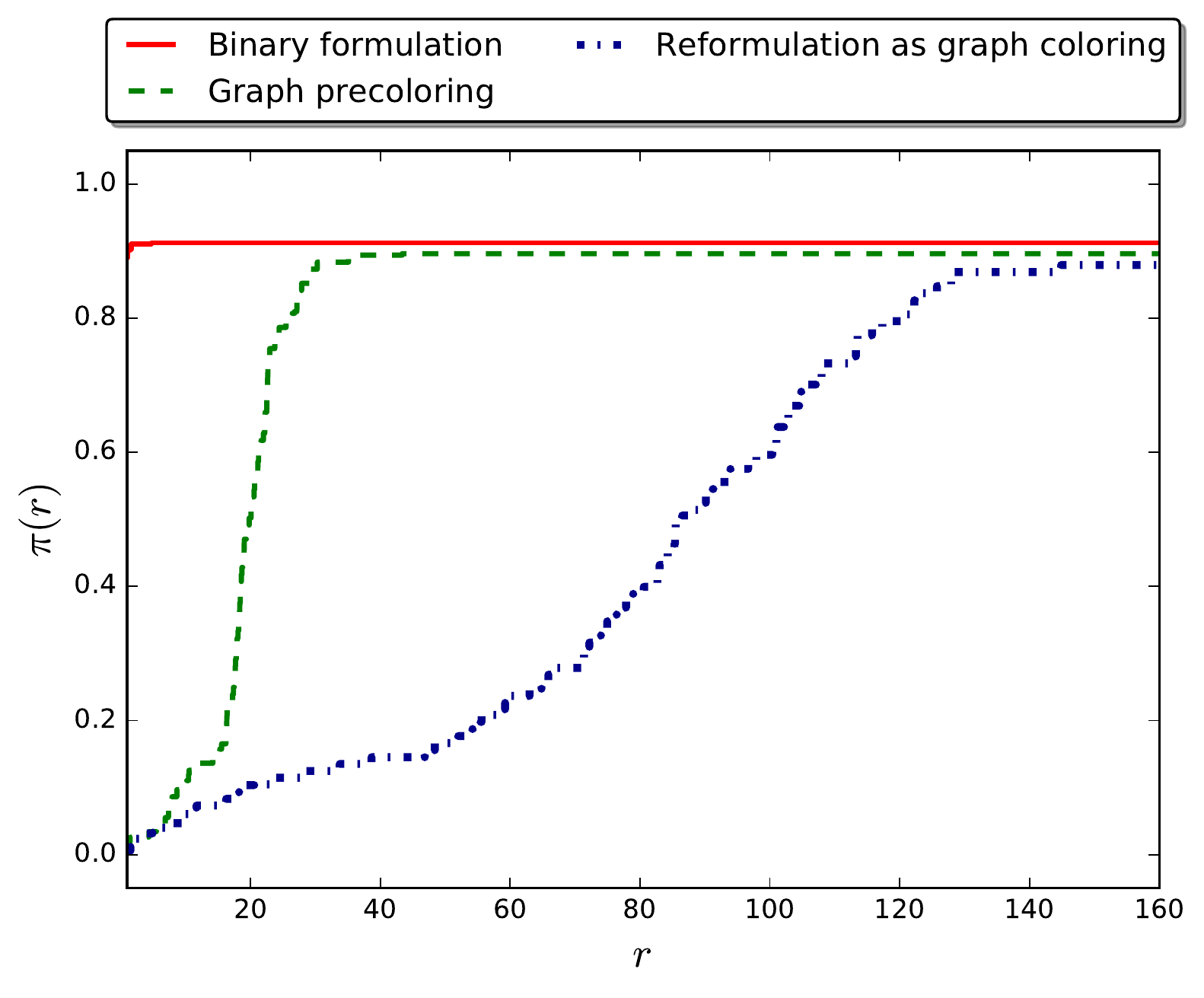}\vspace{5pt}
	\centering
	\begin{tabular}{|c||c|c||c|c||c|c|}
		\cline{2-7}
		\multicolumn{1}{c|}{} & \multicolumn{2}{c||}{Binary} & \multicolumn{2}{c||}{Graph} & \multicolumn{2}{c|}{Reformulation as} \tabularnewline
		\multicolumn{1}{c|}{} & \multicolumn{2}{c||}{formulation} & \multicolumn{2}{c||}{precoloring} & \multicolumn{2}{c|}{graph coloring} \tabularnewline
		\hhline{-||--||--||--}
		Time & Inst. & Cumul. & Inst. & Cumul. & Inst. & Cumul.  \tabularnewline
		\hhline{=::======}
0-49 & 1,688 & 88.8\% & 1,451 & 76.4\% &  261 & 13.7\% \tabularnewline
50-99 & 19 & 89.8\% & 173 & 85.5\% & 534 & 41.8\% \tabularnewline
100-149 & 15 & 90.6\% & 40 & 87.6\% & 451 & 65.6\% \tabularnewline
150-199 & 6 & 90.9\% & 22 & 88.7\% & 267 & 79.6\% \tabularnewline
200-249 & 4 & 91.2\% & 12 & 89.4\% & 118 & 85.8\% \tabularnewline
250-299 & 2 & 91.3\% &5 & 89.6\% & 45 & 88.2\% \tabularnewline
\hhline{-||--||--||--}
Unsolved & 166 & 100\% & 197 & 100\% & 224 & 100\% \tabularnewline
\hhline{-||------}
	\end{tabular}
	\caption{Time spent (in seconds) to find the solution of 95 different Sudoku problems by DR with the graph precoloring, the binary, and the graph coloring formulations. For each problem, 20~starting points were randomly chosen. We stopped the algorithm after a maximum time of 5~minutes, in which case the problem was labeled as ``Unsolved''. The results are shown in a table and a performance profile.}\label{fig:Sudoku}
\end{figure}

In Table~\ref{tbl:Sudoku_fails} we list the Sudokus for which either the binary or the graph precoloring formulation failed to find a solution for some starting point. It is apparent that the three methods tend to fail on the same Sudokus. The reformulation as graph coloring was clearly the most successful method for Sudoku~19. The graph precoloring formulation had a very bad performance on Sudoku~22, compared to the other two methods. On the other hand, it is remarkable that the binary formulation was significantly less successful than the graph precoloring for Sudoku~90, and that it failed to find any solution at all for Sudoku~25. Both the graph precoloring formulation and the reformulation as graph coloring also had troubles with this Sudoku, and were only able to find a solution for 3  and 2 out of the 20 starting points, respectively. This Sudoku is the one shown in Figure~\ref{fig:Sudoku_colored}.

\begin{table}[ht!]
\centering
\begin{tabular}{|c|c|c|c|c|c|c|c|c|c|}
\hline
Sudoku Number & 5 & 12 & 13 & 17 & 19 & 22 & 25 & 29 & 38 \\
\hline\hline
Binary formulation & {0} & {0} & {16} & 19 & 5 & {1} & 20 & {1} & 17 \\
Graph precoloring & 6 & 1 & 18 & {18} & 13 & 19 & {17} & 7 & 15\\
Reformulation as graph coloring & 13 & 1 & {16} & {18} & {1} & 9 & 18 & 15 & {12}\\
\hline
\end{tabular}\vspace{8pt}
\begin{tabular}{|c|c|c|c|c|c|c|c|c|c|}
\hline
Sudoku Number & 53 & 59 & 66 & 82 & 83 & 85 & 86 & 90 & 94 \\
\hline\hline
Binary formulation& {0} & {0} & 14 & {0} & 5 & 18 & 17 & 14 & 19 \\
Graph precoloring  & 5 & 3 & 13 & 3 & 5 & 15 & {15} &{ 8} & 16 \\
Reformulation as graph coloring & 6 & 1 & {11} & 7 & {4} & {14 }& 16 & 14 & {15}\\
\hline
\end{tabular}
\caption{Number of failed runs in either the binary or the graph precoloring formulation. Sudokus not listed here where successfully solved by these two formulations for every starting point (not all the Sudokus where the graph coloring reformulation failed are listed).}
\label{tbl:Sudoku_fails}
\end{table}

Finally, in our last experiment, we explored the behavior of DR for solving the knight's tour problem when the size of the chessboard is increased. Results are displayed in Figure~\ref{fig:Exp_KnTour}, where we analyze both paths and cycles with the two formulations $\widetilde C_1\cap\widetilde C_{3,\text{odd}}\cap\widetilde C_{3,\text{even}}$ (red crosses) and $\widetilde C_1\cap \widetilde C_2\cap\widetilde C_{3,\text{odd}}\cap\widetilde C_{3,\text{even}}$ (blue dots). Clearly, the formulation including the redundant constraint $\widetilde C_2=\mathbb{R}^{n\times n}$ is much faster. For this reason, no knight's paths of size $10$ or $11$ are shown for the formulation without $\widetilde C_2$, as the algorithm was stopped before it had enough time to converge. The rate of success of both formulations for paths and cycles was very similar, around 95\%.
It can be observed an exponential dependence between time and size, which makes DR to be inappropriate for big puzzles. It is remarkable that the line~$t(n)$ obtained by linear regression predicts an average time of $t(12)=$ 1,439~seconds for finding a knight's cycle in a $12\times 12$ chessboard, and this totally fits with the average time of 1,397~seconds obtained in the experiment shown in Figure~\ref{fig:Knights_Tour}.

\begin{figure}[ht!]
	\centering
	\subfigure[Knight's paths.]{\includegraphics[width=0.49\linewidth]{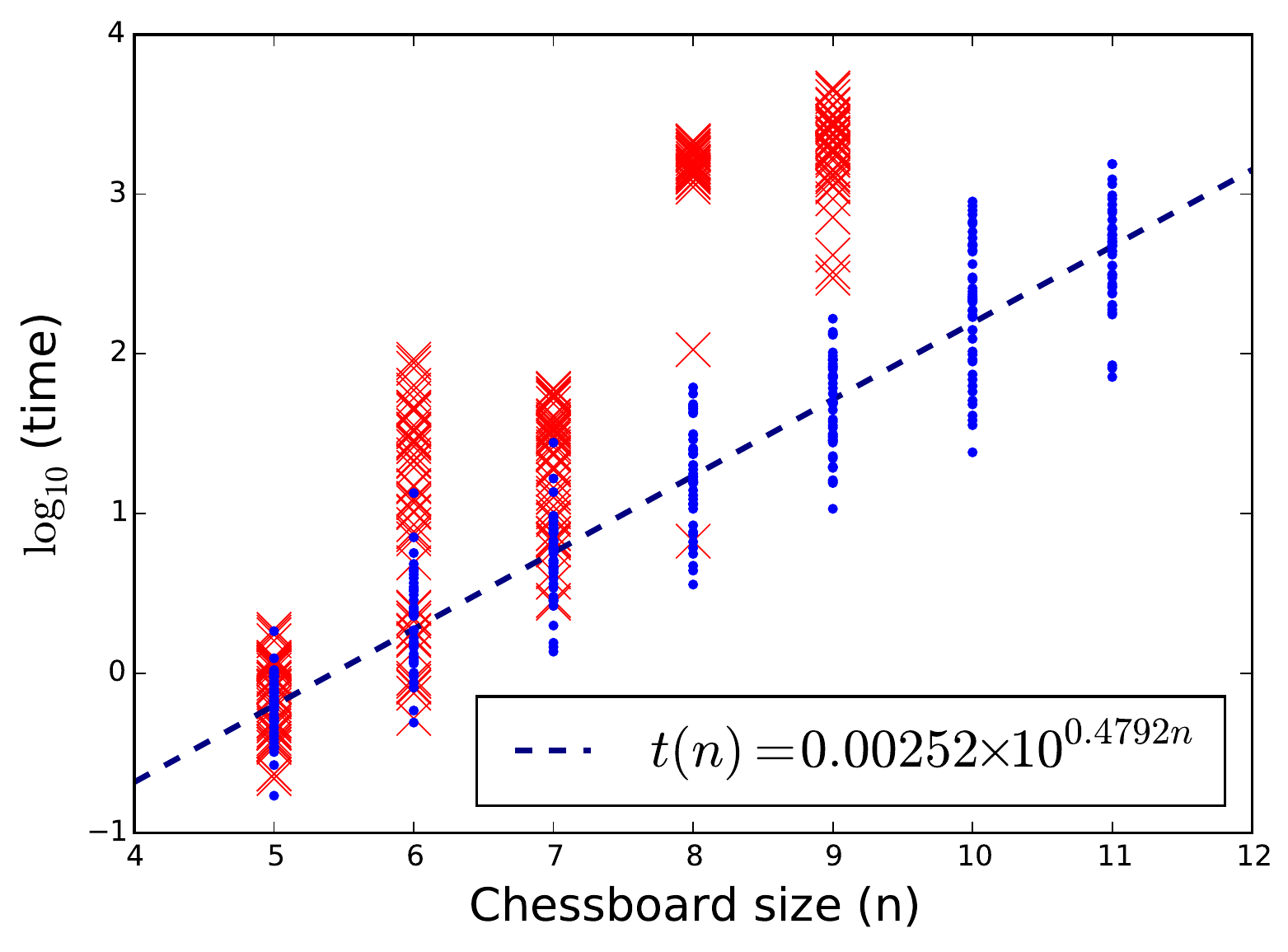}}\hfill
	\subfigure[Knight's cycles.]{\includegraphics[width=0.49\linewidth]{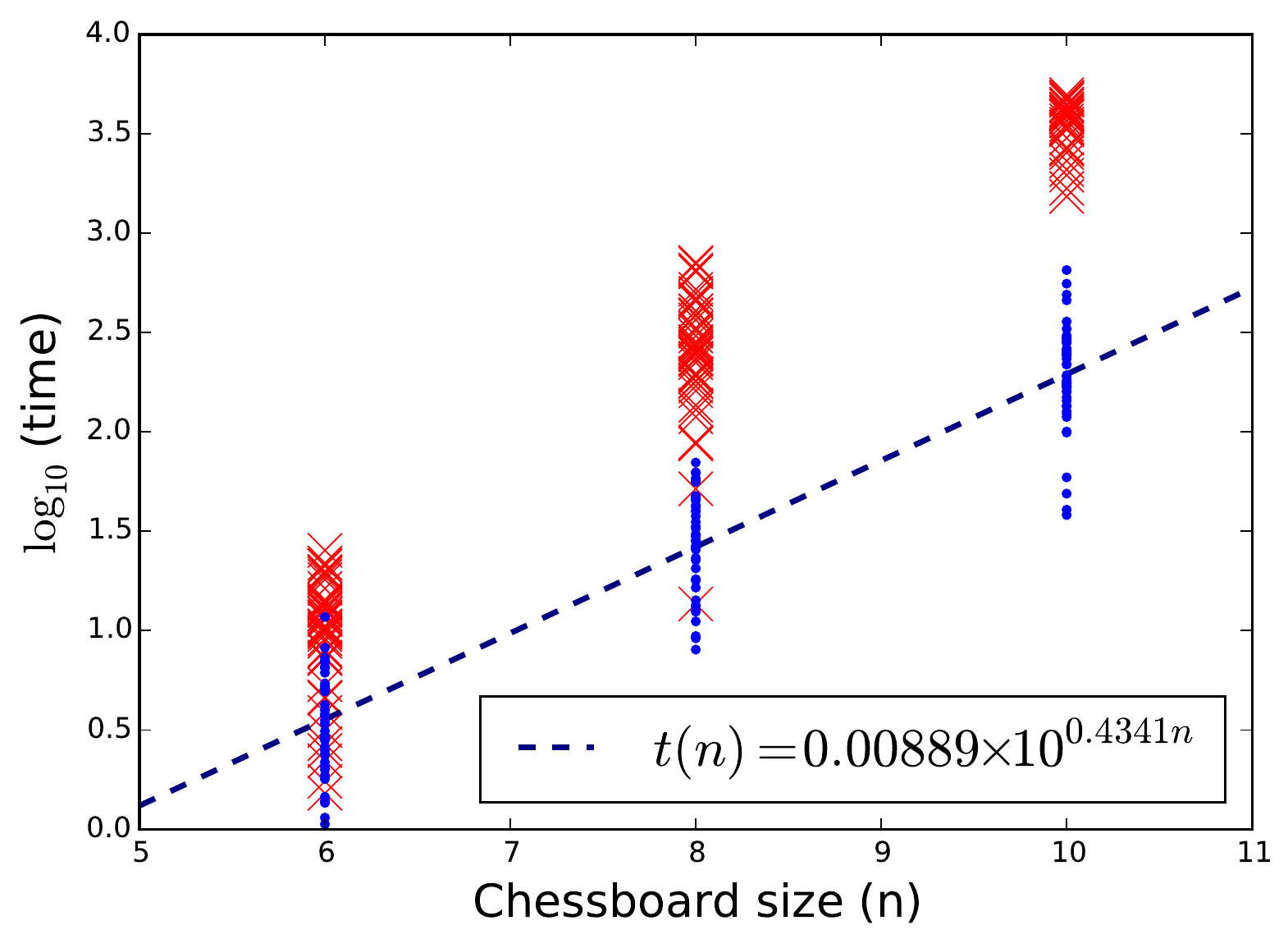}}
	\caption{Time (in $\log_{10}$) required by DR for finding open and closed knight's tours on chessboards of different size. For each size, 50 random starting points were chosen. Blue dots represent instances of the DR method applied with the addition of the redundant constraint $\widetilde C_2=\mathbb{R}^{n\times n}$, while red crosses represent instances where the method was run without $\widetilde C_2$. The dotted lines were obtained by linear regression. The algorithm was stopped after a maximum time of 5,000 seconds, in which case the instance is not displayed.}
	\label{fig:Exp_KnTour}
\end{figure}

\section{Conclusion}\label{sec:conclusion}
We showed that the Douglas--Rachford method can be used as a successful heuristic for solving graph coloring problems. A wide collection of examples and variants of these problems were considered along the paper: precoloring and list coloring problems (including Sudoku puzzles), 3-SAT problems, 8-queens puzzles and generalizations, and Hamiltonian path problems (as the knight's tour problem).

A key aspect for the success of the method was to formulate the problems as suitable combinatorial feasibility problems. In this framework, the Douglas--Rachford method had already been proved to be an effective heuristic~\cite{ABTmatrix,ABTcomb,Elser}, despite the shortfall of theoretical results that justify its good performance.

We tested the performance of Douglas--Rachford for solving a representative sample of graph coloring problems. It is important to point out that the Douglas--Rachford algorithm is conceptually simple and easy to implement. For simple graphs, the method was able to find a solution for almost every random starting point. For more complex problems, we showed the importance of adding maximal clique information for the success of the method. It is worth mentioning the results in the 3-SAT experiment, where we observed that the use of maximal clique information was decisive.

As expected, in problems where it was possible to successfully apply Douglas--Rachford to the original problem, the method became slower when it was applied to the reduction of the problem to a graph coloring problem. This is the case for Sudoku puzzles, which were solved in our experiments much faster when the method was applied to the formulation of the problem as a binary feasibility problem (on average, 6 and 20 times faster than the graph precoloring formulation and the reformulation as graph coloring, respectively). Nevertheless, it was interesting to observe that the rate of success in finding the solution was high and very similar for the three formulations.

For the knight's tour problem, we showed a clear exponential dependence of the time needed to find a solution with respect to the size of the chessboard. After all, this is not that surprising, due to the NP-completeness of the problem. This shows that the Douglas--Rachford method is probably inadequate for tackling big complex graphs.

In the convex setting, for infeasible problems, the sequence generated by Douglas--Rachford provably tends to infinity (in norm). In our experiments, we obtained some similar results for some particular graphs (see Figure~\ref{fig:infeasible}), a behavior that seems to be strongly influenced by the formulation of the feasibility problem.

All the above motivates us to further study in future research why the Douglas--Rachford algorithm can successfully solve this type of nonconvex problems, as well as analyze the detection of infeasibility in nonconvex settings with this algorithm.

\paragraph{Dedication and acknowledgements}

This paper is dedicated to the memory of Jon Borwein, for his enthusiastic comments and suggestions on a very preliminary version of the manuscript. Jon was planning to collaborate with us on this paper after getting back from Canada. Unfortunately, he stayed there forever. We greatly missed his valuable input in the elaboration of this work, and we will surely miss him in the future.\newline

F.J. Arag\'on and R. Campoy were partially supported by MINECO of Spain  and  ERDF of EU, grant MTM2014-59179-C2-1-P. F.J. Arag\'on was supported by the Ram\'on y Cajal program by MINECO of Spain  and  ERDF of EU (RYC-2013-13327) and R. Campoy was supported by MINECO of Spain and ESF of EU (BES-2015-073360) under the program ``Ayudas para contratos predoctorales para la formaci\'on de doctores 2015''.

\begin{figure}[h!]
	\centering
	\subfigure[Knights on a classic chessboard.]{\includegraphics[width=0.46\linewidth]{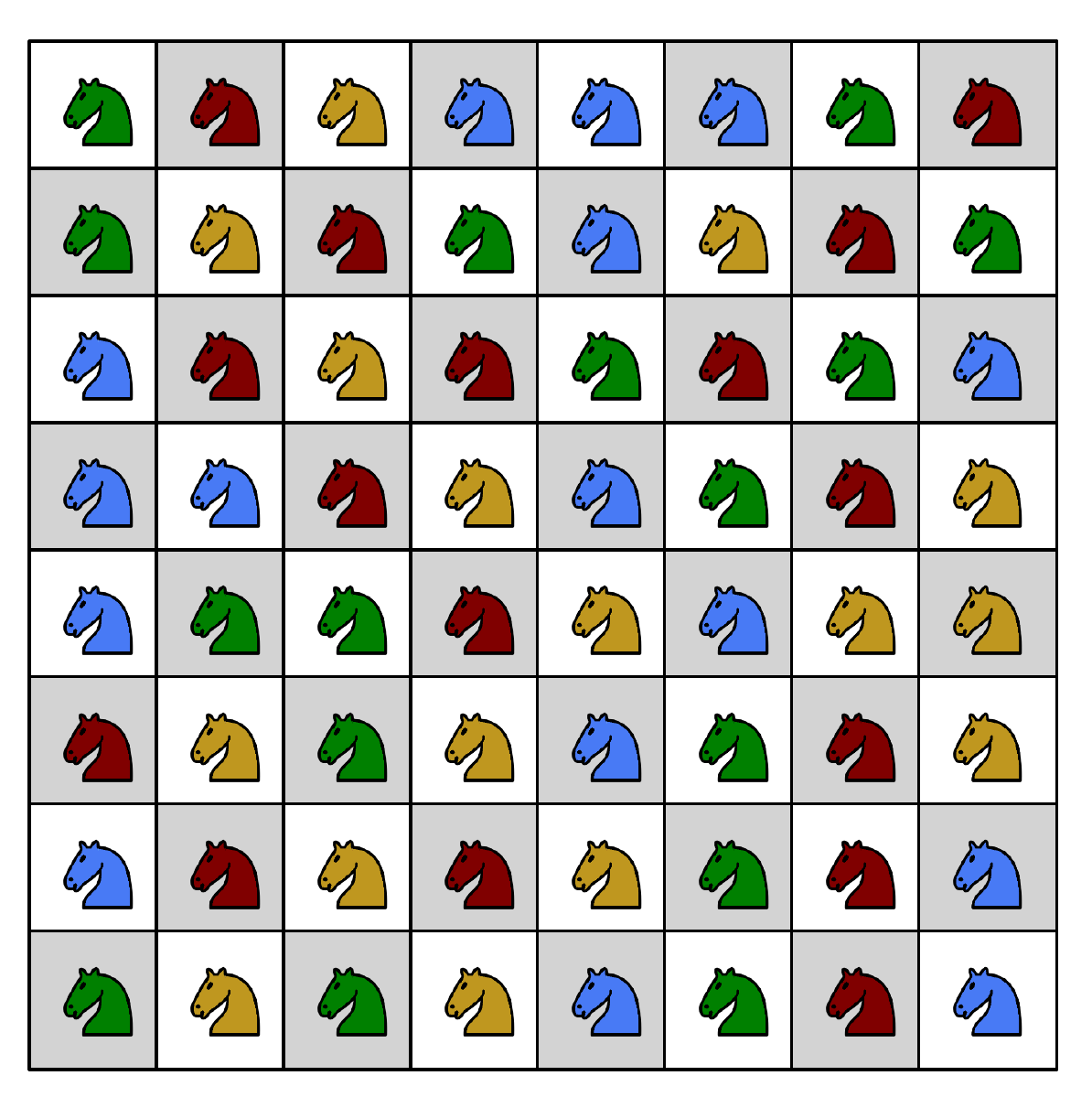}}\quad	
	\subfigure[10-queens puzzle in a chessboard with a hole.]{\includegraphics[width=0.46\linewidth]{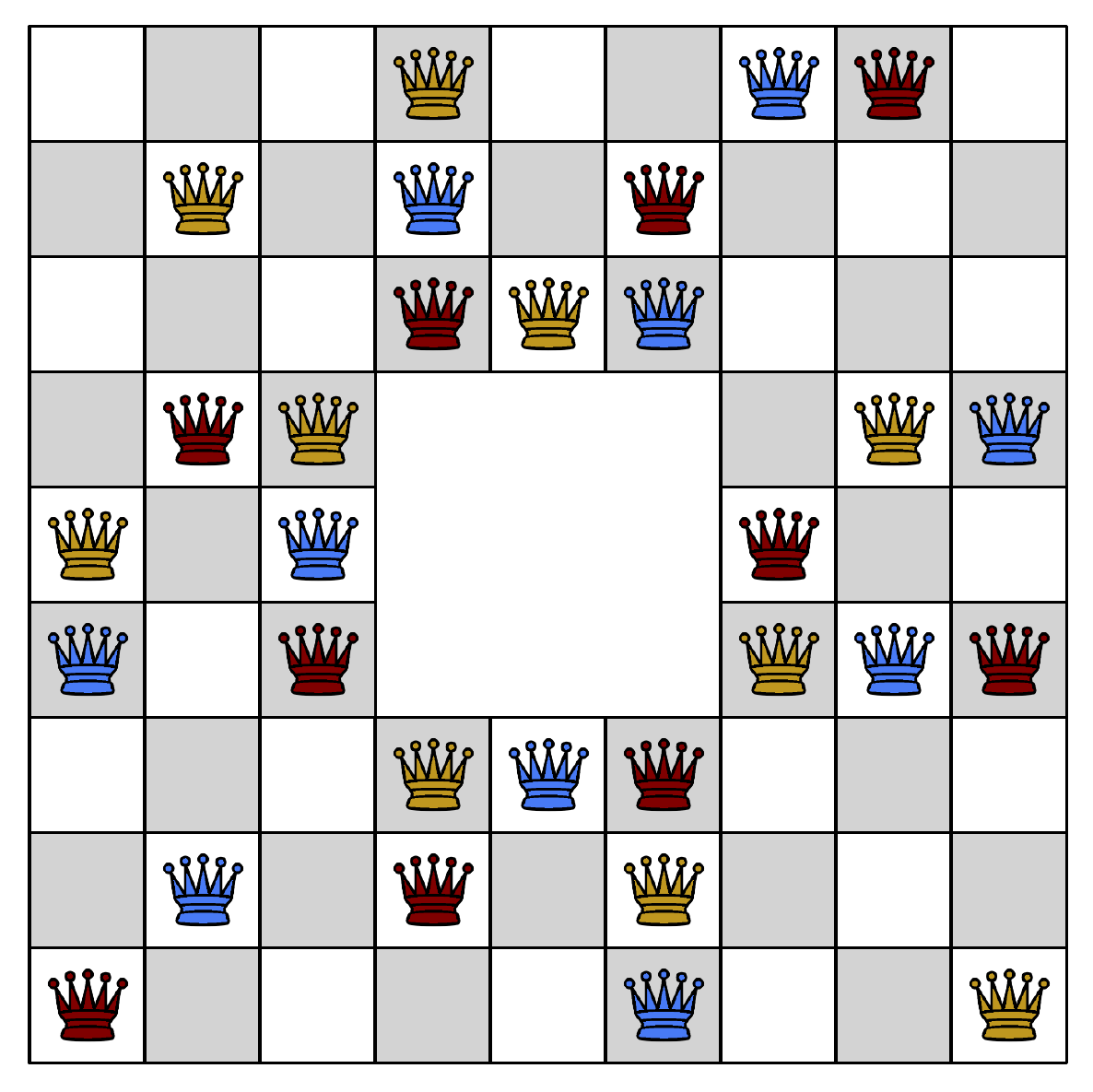}}
	\subfigure[`$\pi$-zzle'.]{\includegraphics[width=0.6\linewidth]{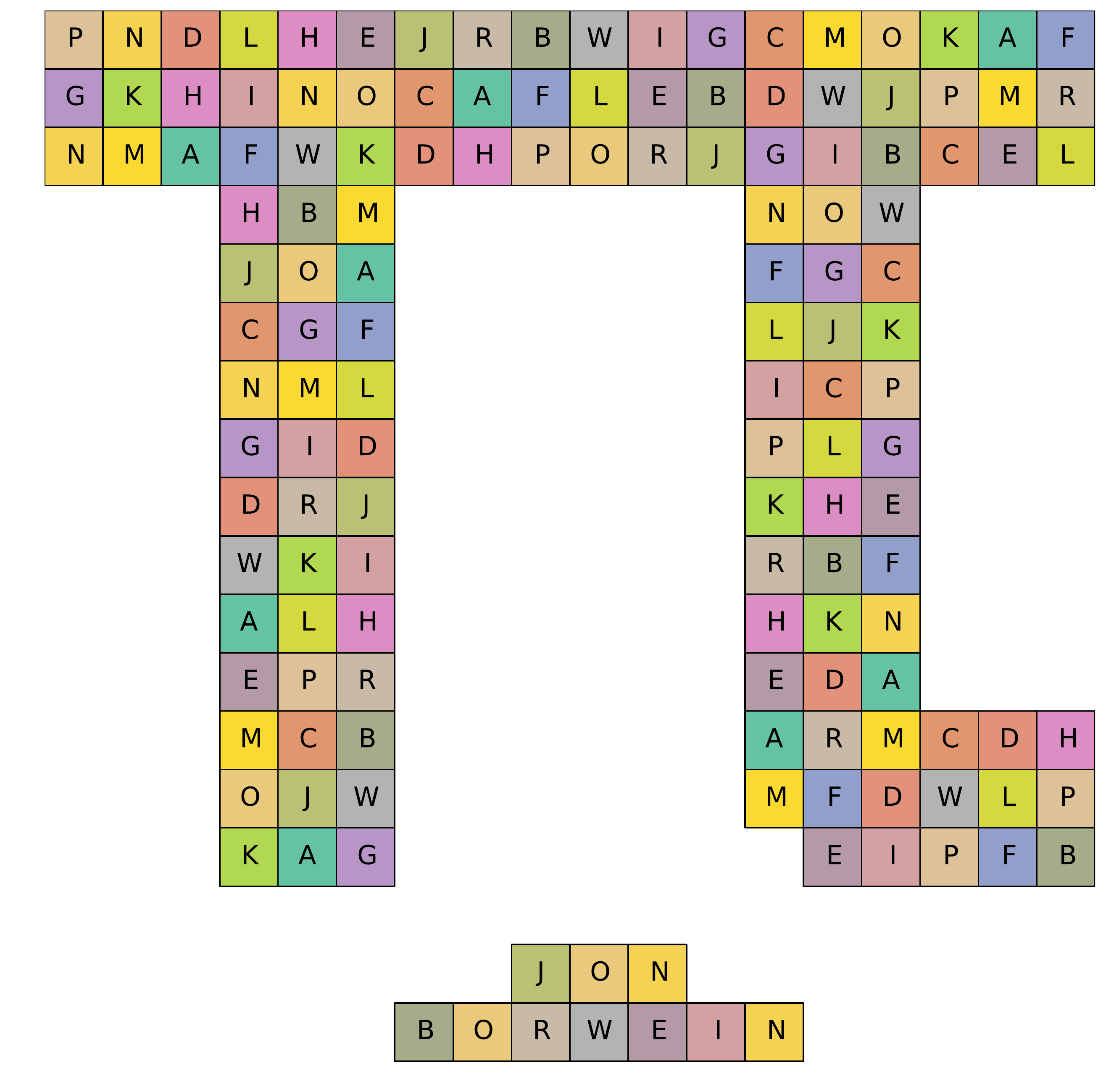}}
	\caption{Solution to the puzzles in Figure~\ref{fig:puzzles} computed with DR.  For 10~random starting points, the average (maximum) time spent for puzzles (a), (b) and (c) was 0.23, 3.32 and 252.82 seconds (0.35, 11.49 and 424.67 seconds), respectively.}\label{fig:puzzles_solved}
\end{figure}

\end{document}